\pdfoutput=1 
\documentclass[12pt,reqno]{amsart}

\usepackage[utf8]{inputenc}
\usepackage[]{tikz}
\usepackage[a4paper]{geometry}
\usepackage{hyphenat}
\usepackage[english]{babel}

\usepackage{amsthm}
\usepackage{amssymb}
\usepackage{amsmath}
\usepackage{amsaddr} 
\usepackage{mathtools}
\usepackage{calc}
\usepackage{comment}
\usepackage{nicefrac}

\usepackage{listliketab}

\usepackage{hyperref}

\newtheorem{defi}{Definition}
\newtheorem{theorem}[defi]{Theorem}
\newtheorem{lem}[defi]{Lemma}
\newtheorem{fact}[defi]{Fact}
\newtheorem{opqn}[defi]{Open Question}
\newtheorem{prop}[defi]{Proposition}
\newtheorem{kor}[defi]{Corollary}

\def\IN{\mathbb{N}} 

\def\SigmaS{\Sigma^{\ast}} 
\def\SigmaN{\Sigma^{\IN}} 
\def\prefix{\sqsubset} 
\def\prefixeq{\sqsubseteq} 
\def\suffixeq{\sqsupseteq} 
\def\truepath{\mathcal{T}} 

\begin{document}

\allowdisplaybreaks

\title[Benign approximations and non-speedability]{Benign approximations and non-speedability\\{\scriptsize (Draft; \MakeUppercase\today)}}

\author{Rupert H\"olzl and Philip Janicki}
\address{University of the Bundeswehr Munich}
\email{r@hoelzl.fr}
\email{philip.janicki@unibw.de}

\begin{abstract}
A left-computable number $x$ is called  {\em regainingly approximable} if there is a computable increasing sequence $(x_n)_n$ of rational numbers converging to $x$ such that $x - x_n < 2^{-n}$ for infinitely many $n \in \IN$; and it is called {\em nearly computable} if there is such an~$(x_n)_n$ such that for every computable increasing function $s \colon  \IN \to \IN$ the sequence ${(x_{s(n+1)} - x_{s(n)})_n}$ converges computably to~$0$.
In this article we study the relationship between both concepts by constructing on the one hand a non-computable number that is both regainingly approximable and nearly computable, and on the other hand a left-computable number that is nearly computable but not regainingly approximable; it then easily follows that the two notions are  incomparable with non-trivial intersection. With this relationship  clarified, we then hold the keys to answering an open question of Merkle and Titov: they studied {\em speedable} numbers, that is, left-computable numbers whose approximations can be sped up in a certain sense, and asked whether, among the left-computable numbers, being Martin-Löf random is equivalent to being non-speedable. As we show that the concepts of speedable and regainingly approximable numbers are equivalent within the nearly computable numbers, our second construction provides a negative answer.
\end{abstract}

\maketitle

\section{Introduction}

Let $(x_n)_n$ be a sequence of real numbers. We say that $(x_n)_n$ is \emph{increasing} if we have $x_{n} < x_{n+1}$ for all $n \in \IN$. If we have $x_{n} \leq x_{n+1}$ for all $n \in \IN$, we say that $(x_n)_n$ is \emph{non-decreasing}. Analogously we define a \emph{decreasing} and a \emph{non-increasing} sequence. A real number is called \emph{left-computable} if there exists a computable non-decreasing (or, equivalently, increasing) sequence of rational numbers converging to it.
The left-computable numbers play an important role in computable analysis and in the theory of algorithmic randomness.

A real number $x$ is called \emph{computable} if there exists a computable sequence $(x_n)_n$ of rational numbers satisfying $\left|x - x_n\right| < 2^{-n}$ for all $n\in\IN$. It is easy to see that every computable number is left-computable, but the converse is not true. It is also easy to see that a real number~$x$ is computable if and only if there exists a computable non-decreasing sequence $(x_n)_n$ of rational numbers with $x - x_n < 2^{-n}$ for all $n \in \IN$. However, if we only require that the condition $x - x_n < 2^{-n}$ be satisfied for infinitely many $n \in \IN$, we obtain a different subset of the left-computable numbers introduced by Hertling, Hölzl, and Janicki~\cite{HHJ2023}.

\begin{defi}[Hertling, Hölzl, Janicki~\cite{HHJ2023}]
	A real number $x$ is called \emph{regainingly approximable} if there exists a computable non-decreasing sequence of rational numbers $(x_n)_n$ converging to $x$ with $x - x_n < 2^{-n}$ for infinitely many $n \in \IN$.
\end{defi}

Obviously every computable number is regainingly approximable, but the converse is not true~\cite{HHJ2023}. 
Thus, with respect to inclusion, the regainingly approximable numbers are  properly lodged between the computable and the left-computable numbers. 

For $A \subseteq \IN$, we define a real number in the interval $\left[0, 1\right]$ via $x_A := \sum_{n \in A} 2^{-(n+1)}$. Then clearly $A$ is computable if and only if $x_A$ is computable. If $A$ is only assumed to be computably enumerable, then $x_A$ is a left-computable number; the converse is not true as pointed out by Jockusch (see Soare~\cite{Soa69a}). We say that a real number $x \in \left[0,1\right]$ is \emph{strongly left-computable} if there exists a computably enumerable set $A \subseteq \IN$ with $x_A = x$.

\begin{defi}\label{def:konvergenzmodul}
	Let  $(x_n)_n$ be a convergent sequence and $x := \lim_{n\to\infty}  x_n$.
	\begin{enumerate}
		\item A function $f\colon \IN \to \IN$ is called a \emph{modulus of convergence} of $(x_n)_n$ if for all~${n \in \IN}$ and for all $m \geq f(n)$ we have $\left| x - x_m\right| < 2^{-n}$.
		\item We say that the sequence $(x_n)_n$ \emph{converges computably} to $x$ if it has a computable modulus of convergence.
	\end{enumerate}
\end{defi}

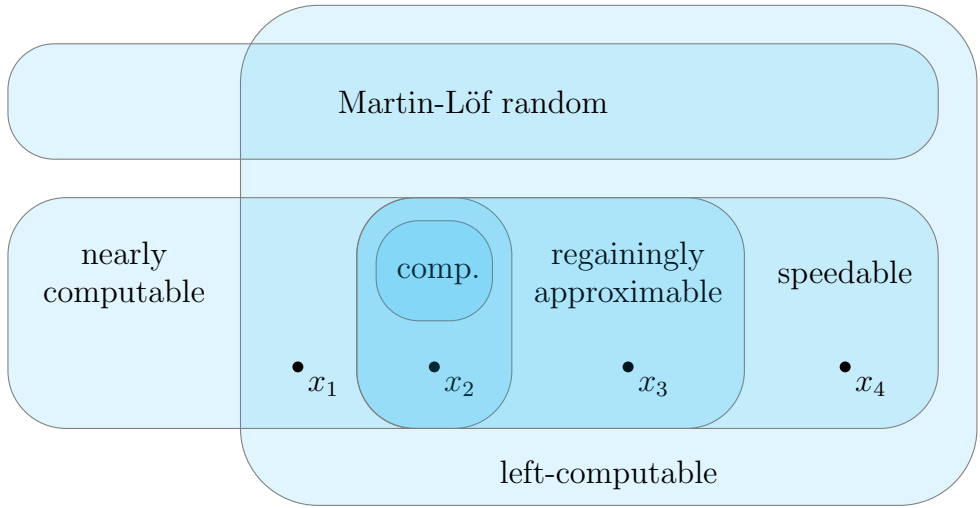
\begin{figure}\label{kldhfkjradfrttsgfrtsdsf}
	\scalebox{1.02}{
		\begin{tikzpicture}[set/.style={fill=cyan,fill opacity=0.125, draw=gray},scale=2,point/.style = {circle, fill=black, inner sep=0.05cm,
				node contents={}}]
			
			\draw[set, rounded corners=1cm] (0, 0.25) rectangle (4.75, 3.5) {};
			\node at (2.375,0.45) {left-computable};
			
			\draw[set, rounded corners=0.6cm] (-1.5, 2.5) rectangle (4.5, 3.25) {};
			\node at (1.5,2.875) {Martin-Löf random};
			
			\draw[set, rounded corners=0.75cm] (-1.5, 0.75) rectangle (1.75, 2.25) {};
			\node[align=center] at (-0.75,1.75) {nearly\\computable};
			\node at (0.37,1.15) [point, label={[label distance=-0.1cm]below right:$x_1$}];
			\draw[set, rounded corners=0.55cm] (0.875, 1.45) rectangle (1.625, 2.1) {};
			\node[align=center] at (1.25,1.75) {\phantom{.}comp.};
			\node at (1.25,1.15) [point, label={[label distance=-0.1cm]below right:$x_2$}];
			
			\draw[set, rounded corners=0.75cm] (0.75, 0.75) rectangle (3.25, 2.25) {};
			\node[align=center] at (2.5,1.75) {regainingly\\approximable};
			\node at (2.5,1.15) [point, label={[label distance=-0.1cm]below right:$x_3$}];
			\draw[set, rounded corners=0.75cm] (0.75, 0.75) rectangle (4.5, 2.25) {};
			\node[align=center] at (3.9,1.75) {speedable};
			\node at (3.9,1.15) [point, label={[label distance=-0.1cm]below right:$x_4$}];			
	\end{tikzpicture}}
	\caption{An overview of the three different notions of benign approximability studied in this article, and of numbers witnessing their separations: $x_1$ and $x_2$ are constructed in Theorem~\ref{satz:FBNAA} and Theorem~\ref{satz:fast-berechenbar-aufholend-approximierbar}, respectively. Proposition~\ref{sadjasjhkrtjhqwejhxvbbmasdgh} establishes the existence of~$x_3$, and Corollary~\ref{dfsdjhhdbfajfkjgsdjfhdasnfbsafevcvssda} that of $x_4$.}
\end{figure}
It can be shown that a real number is computable if and only if there exists a computable sequence of rational numbers converging computably to it. In a detailed study, Hertling and Janicki~\cite{HJ2023} used the concept of computable convergence to define another interesting subset of the real numbers.
\begin{defi}[Hertling, Janicki~\cite{HJ2023}]\label{def:fast-berechenbar}
	\;
	\begin{enumerate}
		\item A sequence $(x_n)_n$ is called \emph{nearly computably convergent} if it converges and, for every computable increasing function $s \colon  \IN \to \IN$, the sequence $(x_{s(n+1)} - x_{s(n)})_n$ converges computably to $0$.
		\item A real number is called \emph{nearly computable} if there exists a computable sequence of rational numbers which converges nearly computably to it.
	\end{enumerate}
\end{defi}

Naturally, every computable number is nearly computable, but the converse is not true; this follows from a theorem of Downey and LaForte~\cite[Theorem~3]{DL02} when combined with 
a result of Hertling and Janicki~\cite[Proposition~7.2]{HJ2023}, which in turn is based on previous work by Stephan and Wu~\cite[Theorem~5]{SW2005}.

In this article we are interested in the relationship between
the set of left-com\-pu\-table nearly computable numbers and that of regainingly approximable numbers; as every computable number is both nearly computable and regainingly approximable, we are only interested in non-computable numbers. The article consists of five sections. After discussing some preliminaries in the second section, we use the third section to
construct an example of a regainingly approximable number that is nearly computable but not computable. 
In the fourth section we give an example of a left-computable number that is nearly computable but not regainingly approximable. Both examples are built using infinite injury priority constructions.
Finally, in the fifth section, we discuss the relationship of our results to an  open problem stated by Merkle and Titov~\cite{MT2020}; they investigated when and to what extent  computable approximations to left-computable numbers can be accelerated.
\begin{defi}[Merkle, Titov \cite{MT2020}]
	A left-computable number $x$ is called \emph{speedable} if there exists a constant $\rho \in (0, 1) $ and a computable increasing sequence of rational numbers $(x_n)_n$ converging to $x$ such that there are infinitely many $n \in \IN$ with $\frac{x - x_{n+1}}{x - x_n} \leq \rho$.
\end{defi}
Among other results, they gave a direct proof that speedable numbers are not Martin-Löf random\footnote{This was also  implicitly shown  by Barmpalias and Lewis-Pye~\cite[Theorem~1.7]{BLP2017}. For general background on Martin-Löf random numbers, refer to Downey and Hirschfeldt~\cite{DH2010} or Nies~\cite{Nie2009}.} and asked whether the inverse implication holds as well; that is, whether a left-computable number is speedable if and only if it is not Martin-Löf random.
Previous attempts to answer this question failed because so far speedability appears to be a notion that seems hard to work with in effective constructions. But as it turns out, regaining approximability is the key to approaching the problem. Namely, we prove that the concepts of speedable and regainingly approximable numbers are equivalent within the nearly computable numbers; thus, together with the construction of the fourth section, we obtain a negative answer to the open question, our final main result.
Figure~\ref{kldhfkjradfrttsgfrtsdsf} provides an overview of the relationship of the different classes of numbers studied in this article.


\section{Preliminaries}

Let $\Sigma := \{0, 1\}$ be the binary alphabet, let $\Sigma^*$ be the set of all (finite) binary words, and let $\Sigma^\IN$ be the set of all (infinite) sequences over $\Sigma$. Write $\Gamma := \Sigma^* \cup \Sigma^\IN$. For any $\sigma \in \Sigma^*$ and $\tau \in \Gamma$, we say that $\sigma$ is a \emph{prefix} of $\tau$ (written $\sigma \prefixeq \tau$) if there exists an $\rho \in \Gamma$ with $\sigma \rho = \tau$. Obviously, the empty word $\lambda$ is a prefix of every other word. If $\sigma \prefixeq \tau$ and $\sigma \neq \tau$ hold, we write $\sigma \prefix \tau$ and say that $\sigma$ is a \emph{proper prefix} of $\tau$. Restricted to $\Sigma^*$, the prefix relation $\prefixeq$ is a partial order.

In addition, we introduce the following strict order: For any $\sigma, \tau \in \Gamma$ we say that $\sigma$ is \emph{lexicographically smaller} than $\tau$ (written $\sigma <_L \tau$) if there exists a $\rho \in \Sigma^*$ with $\rho 0 \prefixeq \sigma$ and $\rho 1 \prefixeq \tau$. Note that this relation is a strict total order only if it is restricted to $\Sigma^\IN$. We explicitly point out that according to this definition two distinct elements are incomparable if and only if one is a proper prefix of the other.

For every $\sigma \in \SigmaS$, we denote with $\left|\sigma\right|$ the length of $\sigma$. Finally, we define the function $\nu\colon \SigmaS \to \IN$, where $\nu(\sigma)$ is $\sigma$'s position in the length-lexicographic order of all strings in $\SigmaS$.

\begin{defi}\label{def:true-path}
	Let $(\truepath[t])_t$ be a sequence in $\SigmaS$. We recursively define the sequence $\truepath \in \SigmaN$, called the \emph{true path} of $(\truepath[t])_t$, as follows:
	\begin{align*}
		\mathcal{T}(0) &:= \begin{cases}
			0 &\text{if $0 \prefixeq \mathcal{T}[t]$ for infinitely many $t \in \IN$} \\
			1 &\text{otherwise}
		\end{cases} \\
		\mathcal{T}(e+1) &:= \begin{cases}
			0 &\text{if $\mathcal{T}(0) \dots \mathcal{T}(e) 0 \prefixeq \mathcal{T}[t]$ for infinitely many $t \in \IN$} \\
			1 &\text{otherwise}
		\end{cases}
	\end{align*}
\end{defi}
Note that the true path is also defined when the sequence $(\left|\truepath[t]\right|)_t$ is bounded.

\begin{fact}\label{prop:truepath-properties}
	Let $(\truepath[t])_t$ be a sequence in $\SigmaS$ and let $\truepath$ be its true path. Then the following statements hold:
	\begin{enumerate}
		\item For all $\sigma \in \SigmaS$ with $\sigma <_L \mathcal{T}$, there exist only finitely many $t\in \IN$ with $\sigma \prefixeq \mathcal{T}[t]$.
		\item If the sequence $(\left|\truepath[t]\right|)_t$ is bounded from above, then $L:= \limsup_{t\to\infty} \left|\truepath[t]\right|$ is a natural number and we have $\truepath(e) = 1$ for all $e \geq L$.
		\item If the sequence $(\left|\truepath[t]\right|)_t$ is unbounded, then for all $\sigma \prefix \truepath$ there exist infinitely many $t \in \IN$ with $\sigma \prefixeq \truepath[t]$. \qed
	\end{enumerate}
\end{fact}

Finally, we fix some standard enumeration $(\varphi_e)_e$ of all computable partial functions from $\IN$ to $\IN$. As usual, the notation $\varphi_{e}(n)[t]{\downarrow}$ denotes that the $e$-th Turing machine, which computes $\varphi_{e}$, stops after at most $t$ steps on input $n$. In this article we will use a uniform modification of this enumeration by enforcing the implication $\varphi_e(n)[t]{\downarrow} \Rightarrow \varphi_e(n) \leq t$ for all $e, n, t \in \IN$. \label{enumeration_convention} Note that this modification does not change the property of being a standard enumeration.

\begin{defi}\label{def:laengenfunktionen}
	Let 
	 $\ell\colon \IN^2 \rightarrow \IN$, defined for all $(e,t)\in \IN^2$ via
	\begin{equation*}
		\ell(e)[t] := \max(\{l \leq t \mid \forall k < l \colon \varphi_e(k)[t]{\downarrow} < \varphi_e(k+1)[t]{\downarrow}\} \cup \{-1\}),
	\end{equation*}
be the {\em length function for increasing functions}.
\end{defi}
Clearly, for all $e \in \IN$, the sequence $(\ell(e)[t])_t$ is non-decreasing, and it tends to infinity if and only if $\varphi_{e}$ is total and increasing.
\section{Non-trivial intersection}

In this section we will prove the first main result of the article. It is easy to see that every computable number is both nearly computable and regainingly approximable; we will show that the converse is not true. The following easy characterization of the computable numbers among the 
left-computable numbers will be useful in the construction.
\begin{fact}\label{lem:charakterisierung-lc-ec}
	The following statements are equivalent for a left-computable number~$x$:
	\begin{enumerate}
		\item $x$ is computable.
		\item For every computable non-decreasing sequence of rational numbers $(x_n)_n$ there exists a computable increasing function $s \colon \IN \to \IN$ with $x - x_{s(n)} < 2^{-n}$ for all $n \in \IN$.\qed
	\end{enumerate}
\end{fact}
\begin{theorem}\label{satz:fast-berechenbar-aufholend-approximierbar}
	There exists a regainingly approximable number which is nearly computable, but not computable.
\end{theorem}
	We will prove this theorem by constructing a computable non-decreasing sequence of rational numbers $(x_t)_t$ which converges nearly computably. Let $x$ denote its limit. To ensure that $x$ is not computable, it is enough to ensure that there exists no computable subsequence of $(x_t)_t$ satisfying  Fact~\ref{lem:charakterisierung-lc-ec}~(2). To this end, we define  the following infinite list of negative requirements, for~$e \in \IN$:
	\begin{equation*}
		\mathcal{N}_e \colon \varphi_e \text{ total and increasing}  \Rightarrow \left(\exists m \in \IN\right)  \;  x - x_{\varphi_e(m)} \geq 2^{-m}
	\end{equation*}
	On top of that, we need to ensure that the sequence $(x_t)_t$ converges nearly computably. By definition this is the case if for every computable increasing function $s \colon  \IN \to \IN$ the sequence $(x_{s(t+1)} - x_{s(t)})_t$ converges computably to zero. To achieve this, we define the following infinite list of positive requirements, for $e \in \IN$:
	\begin{equation*}
		\mathcal{P}_e \colon  \varphi_e \text{ total and increasing}   \Rightarrow (x_{\varphi_e(t+1)} - x_{\varphi_e(t)})_t \text{ converges computably to } 0.
	\end{equation*}
	
	At first sight, the two types of requirements seem to be in conflict with each other: On the one hand, the approximation of~$x$ will have to try to satisfy negative requirements by performing possibly large jumps; and this can become necessary during very late stages of the construction. On the other hand, it will have to try to satisfy positive requirements by committing at certain moments during the construction that it will never again make jumps larger than certain sizes. We need to carefully balance out these necessities. 
	
	\smallskip
	
	Our approach is inspired by the work of Downey and LaForte~\cite{DL02}.
	Namely, we work with an infinite injury priority construction on an infinite binary tree of strategies~$\sigma \in \Sigma^*$. Such a strategy with $e := \left|\sigma\right|$ will be responsible for satisfying both requirements $\mathcal{N}_e$ and $\mathcal{P}_e$. We define the \emph{parameters} of a strategy as the following four functions from $\SigmaS \times \IN$ to $\IN$, defined for every $\sigma \in \Sigma^*$ and every~$t\in\IN$:
	\begin{itemize}
		\item A \textit{counter} $c(\sigma)[t]$
		\item A \textit{restraint} $r(\sigma)[t]$
		\item A \emph{satisfaction flag} $s(\sigma)[t]$
		\item A \textit{witness} $w(\sigma)[t]$		
	\end{itemize}
	Throughout the construction, we will ensure that the witness and the restraint are  non-decreasing in $t$. While the witness and the satisfaction flag of a strategy are responsible for handling the corresponding negative requirement, the restraint is relevant for fulfilling the positive one. The counter is a mediator between positive requirements with high priority and negative requirements with lower priority when a conflict occurs. If a strategy of low priority would like to make a large jump in order to fulfill a negative requirement, but can only make smaller jumps due to some restraint imposed by a strategy of higher priority, then with the help of the counter the large jump is divided into several smaller jumps that are then scheduled for later execution. The counter keeps track of how many of these small jumps still have to be made before the original large jump is completed. Once that is the case, the corresponding negative requirement is satisfied.
	While executing these smaller jumps one by one, even stricter restraints might get imposed due to even higher priority strategies. In that case, we need to split them again into even smaller jumps, etc.
	
	More precisely, when negative requirements become threatened, we want to react by making a jump to defeat the threat. If there is no restraint in place that prevents us from doing so, we immediately make the jump. Otherwise, if such restraints are in place, we choose the next higher priority strategy having created such a restraint, and make it responsible for executing the many small jumps needed. This will then happen starting from the next stage. In other words, jumps that cannot be executed immediately will be handed upwards in the chain of command to be handled there.
	
	For the positive requirements, ideally, we would like to not make any jumps at all. Of course this is impossible due to the negative requirements. In fact, it is during the handling of positive requirements that we execute the many small jumps that were scheduled due to the negative requirements. Depending on the situation we may either execute a jump right away or pass it upwards in the chain of command where it will get executed by being split into a large number of even smaller jumps.
	
	Finally, we still need to ensure that the constructed number is regainingly approximable. 
	In its most basic form, the idea that we employ to achieve this is that after a jump originating from some negative requirement has been completed at some stage $t \in \IN$ --- that is, all the small jumps into which it was split have ultimately been made --- we then take care that lower priority strategies can from now on only make jumps whose sum is at most $2^{-t}$. This is done by increasing their witness function values accordingly. In the proof, some significant effort will be required to verify that the technical details of the construction have all been arranged in such way that this works out correctly.
	
	\smallskip
	
	After these introductory remarks, we are now ready to formally prove the theorem.	
\begin{proof}
	We use a special pairing function $\left\langle \cdot , \cdot \right\rangle \colon \SigmaS \times \IN \to \IN$ defined by
	\begin{equation*}
		\left\langle \sigma, n \right\rangle := P(\nu(\sigma), n)
	\end{equation*}
	for all $\sigma \in \SigmaS$ and $n \in \IN$, where $P \colon \IN^2 \to \IN$ is Cantor's pairing function defined by $P(m,n) := \frac{1}{2} (m+n)(m+n+1) + n$ for all $m, n \in \IN$. 
	It is easy to see that $\left\langle \cdot , \cdot \right\rangle$ is a computable bijection.
	We will recursively define a computable non-decreasing sequence of rational numbers $(x_t)_t$ starting with $x_0 := 0$. At the same time, 	we will also recursively define four functions ${c, r, s, w \colon \SigmaS \times \IN \to \IN}$ starting with
	\begin{align*}
		c(\sigma)[0]	&:= 0, \\
		r(\sigma)[0]	&:= 0, \\
		s(\sigma)[0]	&:= 0, \\
		w(\sigma)[0]	&:= \nu(\sigma),
	\end{align*}
	for all $\sigma \in \SigmaS$. 
	
	The construction proceeds in stages. At the end of each stage $t\in \IN$, we will fix a finite binary string $\truepath[t]$; we will say that stage~$t$ {\em settles on}~$\truepath[t]$. To define this string during stage $t$ we proceed in 
	at most $t + 1$ substages, in each of which we will decide on one additional bit of $\truepath[t]$. 
	
	More precisely, in every stage $t$, we will begin in substage $0$ by {\em applying} strategy~$\lambda$; and for every substage $e< t$, if $\sigma$ with $|\sigma|=e$ is the strategy applied in it, then we may explicitly choose to {\em apply} $\sigma0$ or $\sigma1$ in the following substage $e +1$. Alternatively, in every substage, we may also decide that we do not need another substage, thereby ending stage $t$ and going to stage $t+1$ directly. At the end of each stage, we choose the last applied strategy as the string $\truepath[t]$ that stage~$t$ settles on.
	
	We define some terminology:
	\begin{itemize}
		\item At some stages $t$ of the construction, we will need to \emph{initialize} strategies~${\tau \in \SigmaS}$. By this we mean that we set 
		\begin{align*}
			c(\tau)[t+1] &:= 0, \\
			s(\tau)[t+1] &:= 0, \\
			w(\tau)[t+1] &:= \nu(\tau) + t + 2.
		\end{align*}
		Note that this influences how the construction will proceed in the {\em subsequent} stage $t+1$.
		
		\item We say that a strategy $\sigma\in \SigmaS$ is {\em threatened at stage $t$} if the  conditions 
		\begin{itemize}
			\item $s(\sigma)[t] = 0$,
			\item $\ell(e)[t] \geq w(\sigma)[t]$ where $e:=|\sigma|$,
			\item $x_{t} - x_{\varphi_{e}(\ell(e)[t])} < 2^{-w(\sigma)[t]}$
		\end{itemize}
		are satisfied. As described above, to defeat this threat, we would like to react by making some large jump.
		
		\item We say that a strategy $\sigma\in \SigmaS$ is \emph{expansionary at stage $t$} if the conditions
		\begin{itemize}
			\item $s(\sigma)[t]=1$,
			\item $\ell(e)[t]\geq 0$ where $e:=|\sigma|$,
			\item $x_{t} - x_{\varphi_{e}(\ell(e)[t])} < 2^{-r(\sigma)[t]}$
		\end{itemize}
		are satisfied. By the first condition, a $\sigma$ can only be expansionary if it was already threatened since its last initialization. Then the third condition means that $\varphi_{e}$ has made some further progress towards looking like a total, increasing function. Due to how we chose the positive requirements, this has repercussions on the sizes of jumps we may still make in the future. 
	\end{itemize}
	
	We complete the description of the construction by saying what it means for $\sigma \in \SigmaS$ to be \emph{applied} in a substage of stage~$t$:
	\begin{enumerate}
		\item Let $e := \left|\sigma\right|$. If we have $e = t$, then this is the last substage, we set
		\begin{equation*}
			x_{t+1} := x_t,
		\end{equation*}
		we initialize all strategies $\tau \in \SigmaS$ with $\sigma <_L \tau$, and terminate stage~$t$. Otherwise we continue with~(2).
		\item \emph{Negative requirement:} If $\sigma$ is not threatened at stage $t$, jump to~(3).
		
		Otherwise we would like to react to the threat by making a jump. First check if there exists a $\gamma \in \Sigma^*$ with $\gamma 0 \sqsubseteq \sigma$ and $r(\gamma)[t+1] \geq w(\sigma)[t]$. 
		
		If there is no such string, then no higher priority restraint prevents us from making the jump we want to make, and we simply set
		\begin{align*}
			x_{t+1} &:= x_{t} + 2^{-w(\sigma)[t]}.
		\end{align*}
		
		If $\gamma$ exists on the other hand, then choose the longest such $\gamma$, and set
		\begin{align*}
			c(\gamma)[t+1] 	&:= \left\langle \sigma, 2^{r(\gamma)[t+1] - w(\sigma)[t]}\right\rangle, \\
			x_{t+1} &:= x_t.
		\end{align*}
		This means that we cannot make the desired jump right away, but that we have scheduled a large number of small jumps to be made in future stages when strategy~$\gamma$ gets applied.
		
		In either case, this will be the last substage of this stage. We set
		\begin{align*}
			s(\sigma)[t+1] &:= 1
		\end{align*}
		to document that this negative requirement has been handled,  initialize all strategies $\tau \in \SigmaS$ with $\sigma <_L \tau$ or $\sigma \prefix \tau$, and terminate stage~$t$.

		\item \emph{Positive requirement:}
		If $\sigma$ is not expansionary at stage $t$, continue with the next substage $e+1$ applying~$\sigma 1$. The ``1'' means that we currently have no reason to believe that $\varphi_e$ is a total and increasing function, and that it thus need not be considered when we are trying to ensure the near computability of~$x$.
		
		Otherwise, if $\sigma$ is expansionary at stage $t$, we check if $c(\sigma)[t] = 0$. If this is the case, then we set
		\begin{equation*}
			r(\sigma)[t+1] := r(\sigma)[t] + 1
		\end{equation*} 
		and continue with the next substage $e+1$ applying~$\sigma 0$. The ``0'' means that we consider $\varphi_e$ a viable candidate for being a total and increasing function, and that in response we have made our restraint on future jump sizes stricter.
		
		Otherwise, if $\sigma$ is expansionary at stage $t$ but $c(\sigma)[t] > 0$, then we need to delay our reaction as we still have ``homework'' to do first. Namely, we need to execute a list of jumps previously scheduled in an earlier stage.
		
		Formally, in this case, there exist a strategy $\alpha \in \SigmaS$ and a number $k \in \IN$ with $c(\sigma)[t] = \left\langle \alpha, k+1\right\rangle$. Check if there exists a $\gamma \in \Sigma^*$ with $\gamma 0 \sqsubseteq \sigma$. If there is no such string, then we set
		\begin{equation*}
			x_{t+1} := x_{t} + 2^{-r(\sigma)[t]},
		\end{equation*}
		meaning that we have successfully executed one of the scheduled jumps.
	
		If $\gamma$ exists on the other hand, choose the longest such~$\gamma$. We claim that in this case we necessarily have $r(\gamma)[t+1] \geq r(\sigma)[t]$, see Fact~\ref{restrait_monotony} for a justification. Thus we can set
		\begin{align*}
			c(\gamma)[t+1] &:= \left\langle \alpha, 2^{r(\gamma)[t+1] - r(\sigma)[t]}\right\rangle, \\
			x_{t+1} &:= x_t.
		\end{align*}
		This means that instead of executing 
		the scheduled jump right away, we hand it further up in the chain of command, splitting it even further in the process. Note that the label ``$\alpha$'' in the counter, which documents which threat originally made this jump necessary, is maintained in this process.
	
		In either case, this will be the last substage of this stage. We document the fact that one of the scheduled jumps has been taken care of (either by executing it or by delegating it upwards) by setting
		\begin{equation*}
			c(\sigma)[t+1] := \begin{cases}
				0 &\text{if $k = 0$,} \\
				\left\langle \alpha, k\right\rangle &\text{otherwise},
			\end{cases}
		\end{equation*}
		initialize all strategies $\tau \in \SigmaS$ with $\alpha <_L \tau$ or $\alpha \prefix \tau$, and terminate stage~$t$.
	\end{enumerate}
	If some of the parameters $c(\sigma)[t+1]$, $r(\sigma)[t+1]$, $s(\sigma)[t+1]$, or $w(\sigma)[t+1]$ have not been explicitly set to some new value during stage $t$, then we set them to preserve their last respective values $c(\sigma)[t]$, $r(\sigma)[t]$, $s(\sigma)[t]$, or $w(\sigma)[t]$. 
	
	\medskip
	
	We proceed with the verification. Write $J:=\{t\in\IN\colon x_{t+1}>x_t\}$ for the set of stages at which non-zero jumps are made during the construction.

	First note that in this construction, whenever a strategy $\sigma$ is applied at some stage~$t$, then at least all strategies that are lexicographically larger than~$\sigma$ will be initialized at the end of stage~$t$. Furthermore, by construction, whenever a strategy $\sigma$ is initialized, all of its extensions are initialized as well.
	The following three facts are immediate from the definitions; we omit the proofs.	
	\begin{fact}\label{asdjltzhjkqwehk}
		If $\sigma$ is applied and threatened at stage $t$, then stage~$t$  settles on~$\sigma$; similarly, if $\sigma$ is applied and expansionary at stage $t$, then stage~$t$ settles on some extension of~$\sigma$.\qed
	\end{fact}	
		\goodbreak
	\begin{fact}\label{satz:FBNBAA:lem03}
		For all $\sigma, \tau \in \Sigma^*$ and all $t \in \IN$ we have
		\begin{itemize}
			\item \makebox[\widthof{$w(\sigma)[t]$}][r]{$r(\sigma)[t]$} $\leq t$,
			\item \makebox[\widthof{$w(\sigma)[t]$}][r]{$r(\sigma)[t]$} $\leq r(\sigma)[t+1]$,
			\item $w(\sigma)[t]$ $\leq w(\sigma)[t+1]$,
			\item $\sigma \prefixeq \tau \colon  w(\sigma)[t] \leq w(\tau)[t]$.\qed
		\end{itemize}
	\end{fact}
	\begin{fact}\label{fact:counters-on-expansionary-stages}
		Let $t \in \IN$ be a stage when $\sigma$ is applied. Then, for all $\gamma \in \SigmaS$ with $\gamma 0 \prefixeq \sigma$, we have $c(\gamma)[t] = 0$.\qed
	\end{fact}
	
	We establish the claim that was used during the construction.	
	\begin{fact}\label{restrait_monotony}
		For every stage $t$ and $\gamma0 \prefixeq \sigma$, we have $r(\gamma)[t+1] \geq r(\sigma)[t]$.
	\end{fact}
\begin{proof}
	If for some $t'\leq t$ we have $r(\sigma)[t'+1]>r(\sigma)[t']$, then in particular $\sigma$ must have been applied during stage $t'$. But then $\gamma0$ was also applied at stage~$t'$ since $\gamma0 \prefixeq \sigma$. This implies~${r(\gamma)[t'+1]>r(\gamma)[t']}$ by construction. Thus, by induction, 
	\[r(\gamma)[t+1] \geq r(\sigma)[t+1] \geq r(\sigma)[t].\qedhere\]
\end{proof}

\goodbreak

Next, we establish the following two helpful lemmata.

\begin{lem}\label{satz:FBNBAA:lem01}
	Let $m, t_0 \in \IN$ and $\sigma \in \Sigma^*$ with $w(\sigma)[t_0] = m$. There exists at most one stage $t_1 \geq t_0$ with $w(\sigma)[t_1] = m$ at which $\sigma$ is applied and threatened.
\end{lem}
\begin{proof}
	Let $t_1 \geq t_0$ be a stage at which $\sigma$ is applied and threatened. By construction we have $s(\sigma)[t_1]=0$ and $s(\sigma)[t_1+1]=1$. Thus, from now on, $\sigma$ is never threatened again, unless it is initialized at some stage $t_2 > t_1$. But then, by construction, $w(\sigma)[t_2] > w(\sigma)[t_1] = m$. The claim then follows from the fact that $(w(\sigma)[t])_t$ is non-decreasing.
\end{proof}
\begin{lem}\label{satz:FBNBAA:lem02}
	The following statements are equivalent for any $\sigma \in \Sigma^*$:
	\begin{enumerate}
		\item The strategy $\sigma$ is applied and expansionary at infinitely many stages.
		\item Infinitely many stages settle on some extension of $\sigma0$.
	\end{enumerate}
\end{lem}
\begin{proof}
	``$(1) \Rightarrow (2)$'': Suppose that there are infinitely many stages at which $\sigma$ is applied and expansionary. Let $t_0 \in \IN$ be an arbitrary  such stage. 
	We claim that there must be a stage $t^\ast \geq t_0$ where $\sigma0$ is applied. 
	If $c(\sigma)[t_0] = 0$, then $t^\ast=t_0$ by construction. 
Otherwise fix $\alpha \in \SigmaS$ and $k \geq1$ with $c(\sigma)[t_0] = \left\langle \alpha, k\right\rangle$. Then let $t_1, \dots, t_k \in \IN$ denote the $k$ consecutive stages where $\sigma$~is applied and expansionary that immediately follow $t_0$; we claim that  we can find~$t^\ast$ among them. Namely, if between $t_0$ and $t_k$ an initialization of $\sigma$ occurs at some stage $\widehat t$, then let $t^\ast$ be the smallest element of $\{t_0,\dots, t_k\}$ that is greater than~$\widehat t$. Otherwise, if no such initialization occurs, then we must have  $c(\sigma)[t_k] = 0$ by construction, and $t^\ast = t_k$.

\smallskip
	
	``$(2) \Rightarrow (1)$'': Suppose that infinitely many stages settle on some extension of $\sigma0$. By construction $\sigma$ must be applied and expansionary at these stages.
\end{proof}

The following proposition shows that for every $l \in \IN$ there exists a strategy of length~$l$ which is  applied infinitely often but only initialized finitely often. Intuitively speaking, this ensures that for each requirement there is a stable strategy which is applied in infinitely many stages in which it can take care of satisfying the negative and positive requirements. 
\begin{prop}\label{sdfjlw3ljsdadfgjrfasdfdfg}
	For every $l \in \IN$, there exists a binary string $\sigma_l \in \Sigma^{l}$ satisfying the following conditions:
	\begin{enumerate}
		\item There exist infinitely stages in which $\sigma_l$ is applied.
		\item There exist only finitely many stages in which $\sigma_l$ is initialized.
	\end{enumerate}
In particular, the map $t \mapsto \left|\truepath[t]\right|$ is unbounded, and if we let 
$\truepath$ denote the true path of $(\truepath[t])_t$ then  conditions (1) and (2) hold for every $\sigma \prefix \truepath$.
\end{prop}
\begin{proof}
	We proceed by induction. The claim trivially holds for $l = 0$ since $\lambda$ is applied in every stage but never initialized.
	Assume that for some fixed $l \in \IN$ there exists a strategy $\sigma_l \in \Sigma^l$ satisfying both conditions. 
	
	Let $t_0 \in \IN$ be the first stage after the last initialization of $\sigma_l$; in particular, no proper prefix of $\sigma_l$ may be applied and threatened at $t_0$ or later by construction.	
	If there exists a stage $t_1 \geq t_0$ at which $\sigma_l$ is applied and threatened it must be unique; this is because after such a stage we would have ${s(\sigma_l)[t_1 + 1] = 1}$, and since $\sigma_l$ is never initialized again this means that $\sigma_l$ can never be threatened again after $t_1$. If $t_1$ exists then let $t_2 > t_1$ be the next stage at which $\sigma_l$ is applied; otherwise let $t_2$ be the next stage after $t_0$ at which $\sigma_l$ is applied. 
	We show that $\sigma_l0$ cannot be initialized anymore after $t_2$ by excluding the two possible causes:
	\begin{itemize}
		\item {\em $\sigma_l0$ is initialized  when some $\gamma$ is applied and threatened at stage ${t \geq t_2}$:} By construction, $\gamma$ would have to be such that ${\gamma \prefix \sigma_l 0}$ or such that ${\gamma <_L \sigma_l0}$. The case $\gamma=\sigma_l$ is excluded by the choice of $t_2$, and all other choices of $\gamma$ would  lead to an initialization of $\sigma_l$ after $t_2 \geq t_0$ as well, contradiction.
		
		\item {\em $\sigma_l0$ is initialized when some $\gamma$ is applied and expansionary at stage $t \geq t_2$:} 
		We cannot have $\gamma <_L \sigma_l0$ as that would again lead to an initialization of~$\sigma_l$ at stage~$t$ as well. Thus by construction we must have $\gamma 0 \prefixeq \sigma_l$. However, by Fact~\ref{fact:counters-on-expansionary-stages} we must have ${c(\gamma)[t_2] = 0}$ for every such~$\gamma$. As no $\alpha$ with $\alpha <_L \sigma_l$ or $\alpha \prefixeq \sigma_l$ can ever be applied and threatened again after $t_2$, we cannot have $c(\gamma)[t]=\langle \alpha, k\rangle$ for any such $\alpha$ and any $k \in \IN$. 
		Thus, by construction, $\sigma_l0$ cannot be initialized due to any strategy that is applied and expansionary after $t_2$.
	\end{itemize}
	It follows that $\sigma_l0$ is not initialized anymore after stage $t_2$. Thus, in case there are infinitely many stages at which $\sigma_l 0$ is applied, we may choose $\sigma_{l+1} := \sigma_l 0$.
	
	\smallskip
	
	For the other case, assume that there are only finitely many stages at which $\sigma_l 0$ is applied. Then  by the arguments above and by Lemma~\ref{satz:FBNBAA:lem02}
	there is a stage $t_3 \geq t_2$ after which $\sigma_l0$ is never applied again and $\sigma_l$ is never again threatened or expansionary.
	Then after~$t_3$, in each of the (by assumption) infinitely many stages where $\sigma_l$ is applied, $\sigma_l 1$ is applied next. Let $t_4 > t_3$ be the next stage at which $\sigma_l$ is applied; we again exclude the possible causes for initializing $\sigma_l1$ after $t_4$:
	\begin{itemize}
		\item {\em $\sigma_l1$ is initialized  when some $\gamma$ is applied and threatened at stage ${t \geq t_4}$:} By construction, $\gamma$ would have to be such that $\gamma \prefix \sigma_l 1$ or such that $\gamma <_L \sigma_l1$. The case $\gamma=\sigma_l$ is excluded by the choice of $t_2$, the case $\gamma \suffixeq\sigma_l0$ is excluded by the choice of $t_3$, and all other choices of $\gamma$ would  lead to an initialization of $\sigma_l$ after $t_4 \geq t_0$, contradiction.
		
		\item {\em $\sigma_l1$ is initialized when some $\gamma$ is applied and expansionary at stage $t \geq t_4$:} 
		As before we only need to consider $\gamma$ with $\gamma 0 \prefixeq \sigma_l$ and we must have $c(\gamma)[t_4]=0$ for all of them. 
		As no $\alpha$ with $\alpha <_L \sigma_l$ or $\alpha \prefixeq \sigma_l$ or $\sigma_l 0 \prefixeq \alpha$   can ever be applied and threatened again after $t_4$ by the previous arguments,
		we cannot have $c(\gamma)[t]=\langle \alpha, k\rangle$ for any such $\alpha$ and any $k \in \IN$. 
		Thus, by construction, $\sigma_l1$ cannot be initialized due to any strategy that is applied and expansionary after $t_4$.
	\end{itemize}
	Thus, in this case, $\sigma_l1$ is not initialized again after $t_4$, and we choose $\sigma_{l+1} := \sigma_l 1$.

\smallskip

It  trivially follows that $t \mapsto \left|\truepath[t]\right|$ is unbounded. To establish the final claim, for every $l \in \IN$, let $\sigma_l$ be the lexicographically smallest string in $\Sigma^{l}$ satisfying~(1) and~(2). We claim $\sigma_l \prefix \truepath$ for all $l \in \IN$. Suppose otherwise; then, by definition of~$\truepath$, for some $l\in \IN$ there exists a $\tau \in \Sigma^{l}$ with $\tau <_L \sigma$ and satisfying~(1). But then, by construction, every stage in which $\tau$ is applied ends  with an initialization of $\sigma$, which contradicts~(2).
\end{proof}

During the construction, both positive and negative requirements may trigger jumps. However, intuitively speaking, all jumps can be traced back to original threats to negative requirements. This intuition is captured in the following definition and facts, which are easy to verify from the construction.

\begin{defi}
	Define $u \colon J \to \IN$ as follows: For $t \in J$, let $\sigma \in \SigmaS$ be the strategy on which $t$ settles. 
	\begin{enumerate}
		\item If $\sigma$ is applied and threatened at $t$ and there is no $\gamma \in \SigmaS$ for which $\gamma 0 \prefixeq \sigma$ and ${w(\sigma)[t] \geq r(\gamma)[t+1]}$ hold, then let $u(t) := t$.
		\item If $\sigma$ is expansionary at $t$ with $c(\sigma)[t] > 0$ and there is no $\gamma \in \SigmaS$ with $\gamma 0 \prefixeq \sigma$, then there exists a strategy $\alpha \in \SigmaS$ and a number $k \in \IN$ with $c(\sigma)[t] = \left\langle \alpha, k+1\right\rangle$. Then let $u(t)$ be the latest stage $t' < t$ at which~$\alpha$ was applied and threatened.
	\end{enumerate}
\end{defi}
Informally speaking, $u$ maps the stage where a jump is made back to the stage where the threat  at its origin occured. Note that, by construction, for every $t \in J$ one of the two cases above must hold. 

\begin{fact}\label{sdlakjldfgahrjehrfhksbfasjd234}
	For each~$t'\in \IN$ there exist only finitely many $t\in J$ with $u(t)=t'$.\qed
\end{fact}

With the function $u$ defined, we can establish a series of lemmata bounding the total sums of jumps made during the construction.
\begin{lem}\label{satz:FBNBAA:lem:spruenge-expansionary}
	Let $t_1 \in \IN$ and $\sigma \in \SigmaS$ be such that $\sigma$ is applied, not threatened, and expansionary at~$t_1$ with $c(\sigma)[t_1] > 0$. Let $\alpha \in \SigmaS$ be the strategy and $k \in \IN$ be the number with $c(\sigma)[t_1] = \left\langle \alpha, k+1\right\rangle$. Let $t'$ be the last stage before $t_1$ 	
	at which $\alpha$ was applied and threatened. Let $t_2 > t_1$ be the next stage at which $\sigma$ is applied. Suppose that $\sigma$~is not initialized between $t_1$ and $t_2$. Let ${I := \{t \in J \mid t_1 \leq t < t_2 \wedge u(t) = t'\}}$. Then we have
	\begin{equation*}
		\sum_{t\in I} \; (x_{t+1} - x_t) = 2^{-r(\sigma)[t_1]}.
	\end{equation*}
\end{lem}
Note that a strategy $\sigma$ can by definition not be threatened and expansionary at the same stage, so that the condition ``not threatened'' seems superfluous. However, we state and prove the lemma in this form as this will later allow us to reuse the lemma unmodified in the proof of Theorem~\ref{satz:FBNAA}.
\begin{proof}
	We show the statement by induction over the number of zeros in $\sigma$. If $\sigma$~does not contain any zero, then $I$~contains only one element and we have
	\begin{equation*}
		\sum_{t \in I}\; (x_{t+1} - x_t) = x_{t_1 + 1} - x_{t_1} = 2^{-r(\sigma)[t_1]}.
	\end{equation*}
	Otherwise let $\gamma \in \SigmaS$ be the longest string with $\gamma 0 \prefixeq \sigma$ with 
	\[c(\gamma)[t_1+1] = \left\langle \alpha, 2^{r(\gamma)[t_1+1]-r(\sigma)[t_1]} \right\rangle.\] By construction, this is a counter indicating how many stages at which $\gamma$ is applied, not threatened and expansionary have to occur before $\sigma$ can be applied again; by construction we have $r(\gamma)[t] = r(\gamma)[t_1+1]$ for all $t \in I$. As $\gamma$ has one zero less than~$\sigma$, we can assume inductively that the statement already holds for $\gamma$; furthermore note that $\gamma$  cannot be initialized between $t_1$ and $t_2$ by assumption and construction. Thus,
	\[
	\sum_{t \in I}\; (x_{t+1} - x_t) = 2^{r(\gamma)[t_1+1]-r(\sigma)[t_1]} \cdot 2^{-r(\gamma)[t_1+1]} = 2^{-r(\sigma)[t_1]}.\qedhere
	\]
\end{proof}
\begin{fact}\label{remark:spruenge-expansionary-weak}
	Under the same assumptions as in Lemma~\ref{satz:FBNBAA:lem:spruenge-expansionary}, except allowing initializations of $\sigma$ between $t_1$ and $t_2$, we still have 
	$	\sum_{t\in I} \; (x_{t+1} - x_t) \leq 2^{-r(\sigma)[t_1]}$.
\end{fact}
To see this, note that the arguments used in the proof of Lemma~\ref{satz:FBNBAA:lem:spruenge-expansionary} still give the same upper bound, but that initializations might prevent jumps that were scheduled from being executed by resetting counters to zero.
\begin{lem}\label{satz:FBNBAA:lem:spruenge-threatened}
	Let $t_1 \in \IN$ and $\sigma \in \SigmaS$ be such that $\sigma$ is applied and threatened in $t_1$. Let $t_2 > t_1$ and $I := \{t \in J \mid t_1 \leq t < t_2 \text{ and } u(t) = t_1 \}$.
 If $t_2$ is the next stage at which $\sigma$ is applied and $\sigma$ has not been initialized between $t_1$ and $t_2$, then we have $\sum_{t \in I} (x_{t+1} - x_t) = 2^{-w(\sigma)[t_1]}$.
\end{lem}
\begin{proof}
	If there is no longest $\gamma \in \SigmaS$ with $\gamma0 \prefixeq \sigma$ and $r(\gamma)[t_1+1]\geq w(\sigma)[t_1]$, then $I$ contains only one element and we have
	\begin{equation*}
		\sum_{t \in I}\; (x_{t+1} - x_t) = x_{t_1 + 1} - x_{t_1} = 2^{-w(\sigma)[t_1]}.
	\end{equation*}
	Otherwise we have $c(\gamma)[t_1+1] = \left\langle \sigma, 2^{r(\gamma)[t_1+1]-w(\sigma)[t_1]} \right\rangle $. By construction, this is a counter indicating how many stages at which $\gamma$ is applied, not threatened, and expansionary have to occur before $\sigma$ can be applied again. Noting that ${r(\gamma)[t_1+1]=r(\gamma)[t]}$ for all~${t \in I}$ by construction and applying Lemma \ref{satz:FBNBAA:lem:spruenge-expansionary}, we obtain
	\begin{equation*}
		\sum_{t \in I} \;(x_{t+1} - x_t) = 2^{r(\gamma)[t_1+1]-w(\sigma)[t_1]} \cdot 2^{-r(\gamma)[t_1+1]} = 2^{-w(\sigma)[t_1]}.\qedhere
	\end{equation*}
\end{proof}
\begin{fact}\label{remark:spruenge-threatened-weak}
	Under the same assumptions as in Lemma~\ref{satz:FBNBAA:lem:spruenge-threatened}, except allowing $t_2>t_1$ to be arbitrary and allowing initializations of $\sigma$ between $t_1$ and $t_2$, we still have 
	$\sum_{t \in I} (x_{t+1} - x_t) \leq 2^{-w(\sigma)[t_1]}$.
\end{fact}
Note that in this instance we can allow arbitrary $t_2>t_1$ because if an initialization occurs, and $\sigma$ is applied and threatened again afterwards, then the stages at which the jumps resulting from that threat are executed cannot be in $I$ by definition.

\begin{kor}\label{jsjhfarjslkdgssfg}
		Let $t_1, t_2 \in \IN$ with $t_1 < t_2$ and let $I := \{t \in J \mid t_1 \leq t < t_2 \}$. Then,
		\begin{equation*}
			\sum_{t \in I} \; (x_{t+1} - x_t) \leq \sum_{\mathclap{t' \in u(I)}} \; 2^{-w(\truepath[t'])[t']}.\qedhere
		\end{equation*}
\end{kor}
\begin{lem}\label{sdfjkasdjlkfgjkleknjjxvc}
	Let $I \subseteq J$, let $\sigma \in \SigmaS$ and let $t_0 \leq \min u(I)$. Then we have
	\begin{equation*}
		\sum_{\mathclap{t' \in u(I)\colon \truepath[t']=\sigma  }}
	\; 2^{-w(\sigma)[t']} \leq 2^{-w(\sigma)[t_0]+1}.
	\end{equation*}
\end{lem}
\begin{proof}
	First note that by definition of $u$ all the numbers $t'$ that appear in the sum on the left-hand side are stages where $\sigma$ is applied and threatened. But, according to Lemma \ref{satz:FBNBAA:lem01}, between two such stages, $\sigma$ must be initialized, which by construction increases the value of $\sigma$'s witness by at least $2$.
	Thus,	
	\begin{equation*}
		\sum_{\mathclap{t' \in u(I)\colon \truepath[t']=\sigma}}
		\; 2^{-w(\sigma)[t']}
		\leq \sum_{k=0}^{\infty} \;
		2 ^{-(w(\sigma)[t_0]+2k)}
		\leq 2^{-w(\sigma)[t_0]+1}. \qedhere
	\end{equation*}
\end{proof}

From these facts it follows that we indeed construct a left-computable number.
\begin{prop}\label{satz:FBNBAA:prop:konvergenz}
	The sequence $(x_t)_t$ is computable, non-decreasing, and bounded from above. Thus, its limit $x := \lim_{t\to\infty} x_t$ is a left-computable number.
\end{prop}
\begin{proof}
	It is clear that $(x_t)_t$ is computable and non-decreasing. We show that it converges by showing that  $\sum_{t=0}^{\infty} (x_{t+1} - x_t)$ converges. 
	To see this, we apply Corollary~\ref{jsjhfarjslkdgssfg} with $t_1=0$ and $t_2 \to \infty$ as well as Lemma~\ref{sdfjkasdjlkfgjkleknjjxvc}  to obtain 
	\begin{align*}
		\sum_{t=0}^{\infty} \; (x_{t+1} - x_t) 
		&= \sum_{t \in J} \; (x_{t+1} - x_t) \\ 
		&\leq \sum_{\mathclap{t' \in u(J)}} \; 2^{-w(\truepath[t'])[t']} \\
		&= \sum_{\mathclap{\sigma \in \SigmaS}} 
		\;\;\sum_{t' \in u(J)\colon \truepath[t']=\sigma} \!\!\!\!\!\!\!\!\!\!  2^{-w(\sigma)[t']} \\
		&\leq \sum_{\mathclap{\sigma \in \SigmaS}} \; 2^{-w(\sigma)[0]+1} \\
		&= \sum_{\mathclap{\sigma \in \SigmaS}} \; 2^{-\nu(\sigma)+1} \\
		&< \infty.\qedhere
	\end{align*}
\end{proof}

We proceed with some observations that will be helpful later to establish that all requirements are met by the construction.
The next lemma establishes that all jumps scheduled for a strategy on the true path will eventually be executed if we assume that we are already past this strategy's last initialization.
	\begin{lem}\label{satz:FBNBAA:lem07}
	Let $\sigma \prefix \truepath$ and $e:=\left|\sigma\right|$. Let $t_0 \in \IN$ be the earliest stage after which $\sigma$~is no longer initialized. Let $t_1 \geq t_0$ be a stage at which $\sigma$ is applied and threatened, and let $t_2 > t_1$ be another stage at which $\sigma$ is applied. Then we have
	\begin{enumerate}
		\item $x_{t_2} - x_{t_1} \geq 2^{-w(\sigma)[t_1]}$,
		\item $x_{t_2} - x_{\varphi_{e}(\ell(e)[t_1])} \geq 2^{-w(\sigma)[t_1]}$. 
	\end{enumerate}
\end{lem}\goodbreak
\begin{proof} \; \nopagebreak
	\begin{enumerate}
		\item This follows immediately from Lemma~\ref{satz:FBNBAA:lem:spruenge-threatened}.	
		
		\item By the conventions laid out on page~\pageref{enumeration_convention} and since $(x_t)_t$ is non-decreasing, we have ${t_1 \geq  \varphi_{e}(\ell(e)[t_1])}$, thus ${x_{t_1} \geq  x_{\varphi_{e}(\ell(e)[t_1])}}$, and the claim follows immediately from~(1).\qedhere
	\end{enumerate}
\end{proof}
	
	\begin{lem}\label{sdhajhsfdjashfdgjhkfdssdf}
		Let $e \in \IN$ and $\sigma \prefix \truepath$ with $\left|\sigma\right| = e$. Let $t_0 \in \IN$ be the earliest stage after which $\sigma$ is not initialized anymore. Suppose that $\varphi_{e}$ is total and increasing.
		\begin{enumerate}
			\item There exists a uniquely determined stage $t' \geq t_0$ at which $\sigma$ is applied and threatened.
			\item There exist infinitely many stages at which $\sigma$ is applied and expansionary.
		\end{enumerate}
	\end{lem}
\goodbreak
	\begin{proof} \; \nopagebreak
		\begin{enumerate}
			\item Since $\sigma$ is not initialized after $t_0$, we have $w(\sigma)[t] = w(\sigma)[t_0]$ for all $t \geq t_0$. Since $s(\sigma)[t_0] = 0$, there must be an earliest stage $t' \geq t_0$ at which $\sigma$ is applied and threatened. By construction, we have $s(\sigma)[t' + 1] = 1$. Since $\sigma$ is not initialized after $t'$, we have that $\sigma$ cannot be threatened again.
			\item 
			Let $t'$ be as in (1), and let 			
			$t_1$ be an arbitrary number with $t_1 \geq t'$. We claim that there is a stage $t_2 \geq t_1$ satisfying the properties
			\begin{itemize}
				\item $s(\sigma)[t_2] = 1$,
				\item $\ell(e)[t_2] \geq 0$,
				\item $x_{t_2} - x_{\varphi_{e}(\ell(e)[t_2])} < 2^{-r(\sigma)[t_1]}$.
			\end{itemize}
			The first property clearly holds for any stage later than $t'$. The second property must hold for $t_2$ large enough since the assumption that $\varphi_{e}$ is total and increasing implies that $(\ell(e)[t])_t$ tends to infinity. 
			The third condition follows from the same argument together with the fact that $(x_t)_t$ converges.
			
			Finally note that $\sigma$ was by choice of $t_2$ not expansionary at any stage between $t_1$ and $t_2$. By construction this implies that $r(\sigma)[t_1] = r(\sigma)[t_2]$, and thus all conditions for being expansionary are fulfilled by
			$\sigma$ at stage~$t_2$.\qedhere
		\end{enumerate}
	\end{proof}
	
		We are ready to show that all negative requirements are satisfied.
	\begin{prop}\label{satz:FBNBAA:prop:negativ}
		$\mathcal{N}_e$ is statisfied for all $e \in \IN$.
	\end{prop}
	
	\begin{proof}
		Let $e \in \IN$. Suppose that $\varphi_{e}$ is total and increasing. Let $\sigma \prefix \truepath$ with $\left|\sigma\right| = e$, let $t_0 \in \IN$ be the earliest stage after which $\sigma$ is not initialized anymore and let $t_1 \geq t_0$ be the stage at which $\sigma$ is applied and threatened; this stage exists due to Lemma~\ref{sdhajhsfdjashfdgjhkfdssdf}. 
		Then we have  $\ell(e)[t_1] \geq w(\sigma)[t_1]$ in particular. Let $t_2 > t_1$ be the next stage at which $\sigma$ is applied. If we write $m:=w(\sigma)[t_1]$, then, using Lemma~\ref{satz:FBNBAA:lem07},
		\[
			x - x_{\varphi_{e}(m)} = x - x_{\varphi_{e}(w(\sigma)[t_1])} 
			\geq x_{t_2} - x_{\varphi_{e}(\ell(e)[t_1])} 
			\geq 2^{-w(\sigma)[t_1]} 
			= 2^{-m}.
	\]
		Thus, there exists a number $m \in \IN$ with $x - x_{\varphi_{e}(m)} \geq 2^{-m}$, and $\mathcal{N}_e$ is satisfied.
	\end{proof}

	\begin{kor} 
		The number $x$ is not computable.\qed
	\end{kor}

	The following lemma is the main tool for proving that $(x_t)_t$ converges nearly computably.

	\begin{lem}\label{satz:FBNBAA:lem08}
		Let $\sigma \prefix \truepath$ and $t_0 \in \IN$ be the earliest stage after which $\sigma$ is no longer initialized and after which there are no more stages at which $\sigma$ is applied and threatened. Let $t_1, t_2 \in \IN$ with $t_0 \leq t_1< t_2$ be two consecutive stages at which 
		$\sigma$ is applied and expansionary. Then we have
		\begin{equation*}
			x_{t_2} - x_{t_1} \leq 2^{-r(\sigma)[t_1] + 1}.
		\end{equation*}
	\end{lem}
	\begin{proof}
		In order to understand what jumps may be made in stages $t_1$ up to strictly before stage $t_2$, 
		we distinguish between obligations to make jumps that may already stand at the moment when we enter stage $t_1$, and new obligations to make jumps that are created between stages $t_1$ and $t_2$.
		
		First, to quantify jumps due to potential standing obligations, note that in substages $0$ up to $|\sigma|-1$ of stage $t_1$, no jumps are made by construction. Thus, let us consider what jumps may still be scheduled for execution when we are in substage~$|\sigma|$ of stage $t_1$. Consider the following facts:
		\begin{enumerate}
			\item No strategy $\tau <_L \sigma$ will be applied again after stage $t_1$, as otherwise this would lead to initialization of $\sigma$, which contradicts our assumptions. Thus, if any jumps are still scheduled for any such $\tau$ (that is, if $c(\tau)[t_1]>0$), we need not take them into account.
			
			\item For strategies $\tau \in \SigmaS$ with 
			$\tau 0 \prefixeq \sigma$ we must have  $c(\tau)[t_1] = 0$ due to Fact~\ref{fact:counters-on-expansionary-stages}. 	Furthermore, none of these strategies can be applied and threatened anymore from $t_1$ onward, as otherwise that would lead to initialization of $\sigma$, again a contradiction.
		
			\item While for strategies $\tau \in \SigmaS$ with $\tau1 \prefixeq \sigma$ we might have $c(\tau)[t_1]>0$, none of these scheduled jumps will ever be executed from $t_1$ onwards. This is because in this case there needs to exist some $\alpha \suffixeq \tau0$ and some $k \in \IN$ with $c(\tau)[t_1]=\langle \alpha, k \rangle$. When such a jump would be executed, then since $\alpha <_L \sigma$ that would again lead to an initialization of $\sigma$, contradiction. Thus, as before, we need not take such $\tau$'s into account.
			
			\item No $\tau$ with $ \sigma0 <_L \tau$ can be applied at stage $t_1$ by construction. While such $\tau$'s may be applied during stages $t_1 < t < t_2$, note that we set $c(\tau)[t_1+1]=0$ for all of them when they are initialized at the end of stage $t_1$.
			
			\item For a strategy $\tau$ with $\sigma0 \prefixeq \tau$ we may indeed have $c(\tau)[t_1] > 0$.		
			
			\item Similarly, it might be the case that $c(\sigma)[t_1] > 0$.
		\end{enumerate}
		Thus, any jumps made in stages $t_1$ up to $t_2 -1$ are either caused by some strategy as in~(5) or~(6), or must be due to \textit{new} threats to strategies of type~(4) that occur strictly between substage~$|\sigma|$ of stage~$t_1$ and stage $t_2$. We establish upper bounds for all of these cases:
		\begin{itemize}
			\item Concerning strategies of type~(5), by construction, if we make any jump at all while applying them, then these can only occur immediately at stage $t_1$ in substage $|\tau|$.
			We claim that such a jump can not occur due to  $\tau$ being applied and expansionary at $t_1$ with ${c(\tau)[t_1] > 0}$. This is because such a jump would not be made at $t_1$, but would by construction instead be split and scheduled for later execution by some strategy $\rho$ with $\sigma0 \prefixeq \rho0 \prefixeq \tau$, and this strategy $\rho$ will not be applied before stage $t_2$ by construction.
			Thus, the only reason to make a jump in this case is if
			$\tau$ is applied and threatened at $t_1$. Then, by construction the jump made would be of size 
			\[2^{-w(\tau)[t_1]} < 2^{-r(\sigma)[t_1+1]} \leq 2^{-r(\sigma)[t_1]}.\]
			
			\item Concerning~(6), in this case, by construction, we would like to make a jump of size $2^{-r(\sigma)[t_1]}$. The jump may be made immediately, or split and scheduled for later execution by some strategies that are prefixes of $\sigma$; but in either case the total sum of all jumps resulting from this is bounded by $2^{-r(\sigma)[t_1]}$ by Lemma~\ref{satz:FBNBAA:lem:spruenge-expansionary}.
			
			\item Finally, we consider new threats to strategies $\tau$ of type~(4) that occur strictly between substage $|\sigma|$ of stage $t_1$ and stage $t_2$.
			Since these $\tau$'s were all initialized at stage $t_1$, we have $w(\tau)[t_1+1] = \nu(\tau) + t_1 + 2$ for each of them. 
			Thus, applying Lemma~\ref{sdfjkasdjlkfgjkleknjjxvc} to each such $\tau$ and then summing over all of them, the total sum of all jumps made or scheduled when such $\tau$'s are applied in stages between $t_1$ and $t_2$ can be at most $2^{-t_1}$, which by Fact~\ref{satz:FBNBAA:lem03} is at most~$2^{-r(\sigma)[t_1]}$.	
		\end{itemize}
		
		By construction, cases~(5) and~(6) exclude each other, completing the proof.
\end{proof}
		
	After these preparations, we can show that all positive requirements are satisfied.
	\begin{prop}
		$\mathcal{P}_e$ is statisfied 	for all $e \in \IN$.
	\end{prop}
	\begin{proof}
		Let $e \in \IN$. Suppose that $\varphi_{e}$ is total and increasing.	Let $\sigma \prefix \truepath$ with $\left|\sigma\right| = e$. Let $t_0 \in \IN$ be the earliest stage after which $\sigma$ is not initialized anymore. According to Lemma~\ref{sdhajhsfdjashfdgjhkfdssdf} there exist infinitely many stages at which $\sigma$ is applied and expansionary. Now, using~Lemma~\ref{satz:FBNBAA:lem02}, we can infer that there exist infinitely many stages at which $\sigma 0$ is applied. Thus, by construction, the sequence $(r(\sigma)[t])_{t}$ tends to infinity.
		For every $n\in\IN$, write  
		\[t(n):=\min\left\{t \geq t_0\colon\; 
		\parbox{5.8cm}{\centering
			$r(\sigma)[t] \geq n+2$ and $\sigma$ is applied\\
			 and expansionary at stage $t$
		}
		\right\}.\]	
		and then define a function $v \colon  \IN \to \IN$ via $v(n) = \ell(e)[t(n)]$ for all $n\in\IN$. Clearly, $t$ and $v$ are computable. 
		
		We claim that $v$ is a modulus of convergence of the sequence $(x_{\varphi_{e}(t+1)} - x_{\varphi_{e}(t)})_t$. 
		To see this, let $n\in \IN$ and $i \geq v(n)$. Let $t_2 > t_1 \geq t(n)$ be the two consecutive stages at which $\sigma$ is applied and expansionary with $\ell(e)[t_1] \leq i$ and $\ell(e)[t_2] \geq i+1$. Applying Lemma~\ref{satz:FBNBAA:lem08} and the assumption that $\sigma$ is expansionary at $t_1$, we obtain
		\begin{align*}
			x_{\varphi_{e}(i+1)} - x_{\varphi_{e}(i)} &\leq
			x_{\varphi_{e}(\ell(e)[t_2])} - x_{\varphi_{e}(\ell(e)[t_1])} \\
			&= \underbrace{\left( x_{\varphi_{e}(\ell(e)[t_2])} - x_{t_2} \right)}_{\leq 0} + \left( x_{t_2} - x_{t_1} \right) + \left( x_{t_1} - x_{\varphi_{e}(\ell(e)[t_1])} \right) \\
			&\leq  2^{-r(\sigma)[t_1]+1} 
			\makebox[\widthof{$2^{-r(\sigma)[t(n)]+1}$}-\widthof{$2^{-r(\sigma)[t_1]+1}$}]{\,} 
			+ 2^{-r(\sigma)[t_1]} \\
			&\leq  2^{-r(\sigma)[t(n)]+1} + 2^{-r(\sigma)[t(n)]} \\
			&< 
			2^{-n}.\\
		\end{align*}
		Thus $(x_{\varphi_{e}(t+1)} - x_{\varphi_{e}(t)})_t$ converges computably to zero, and $\mathcal{P}_e$ is satisfied.
	\end{proof}
	
	\begin{kor}
		The number $x$ is nearly computable.\qed
	\end{kor}
	
	It remains to show that $x$ is regainingly approximable.
\begin{defi}\label{cut_off_def}
	Let $\sigma \prefix \truepath$ be such that $\varphi_{|\sigma|}$ is total and increasing. Let $t_0 \in \IN$ be the earliest stage after which $\sigma$ is not initialized anymore and let $t' \geq t_0$ be the stage at which $\sigma$ is applied and threatened.
	Then we call $t_\sigma := \max \{t \in J \mid u(t) = t'\}$ the \emph{cut-off stage} for $\sigma$.
\end{defi}

As $\{t \in J \mid u(t) = t'\}$ is non-empty and finite, $t_\sigma$ is well-defined with $t_\sigma \geq t'$.	

\begin{lem}\label{cut_off_properties}
		Let $\sigma \prefix \truepath$ be such that $\varphi_{|\sigma|}$ is total and increasing. Let $t_0 \in \IN$ be the earliest stage after which $\sigma$ is not initialized anymore. Let $t_1 \geq t_0$ be the stage at which $\sigma$ is applied and threatened and let $t_\sigma \geq t_1$ be the cut-off stage for $\sigma$. Then the following statements hold:
		\begin{enumerate}
			\item All strategies~$\tau \in \SigmaS$ with $\sigma \prefix \tau$ or $\sigma <_L \tau$ are initialized at $t_\sigma$.
			\item For all $\tau \in \SigmaS$ we have $c(\tau)[t_\sigma+1] > 0 \Rightarrow  \tau0 <_L \sigma$.
			\item For all $\tau \in \SigmaS$ with $\tau 0 <_L \sigma$ there are no more stages after $t_\sigma$ at which $\tau$ is applied and expansionary. 
			\item If some strategy $\tau$ is applied and threatened at some stage after $t_\sigma$, we must have $\sigma \prefix \tau$ or $\sigma <_L \tau$.
			\item We have $x - x_{t_\sigma+1} \leq 2^{-(t_\sigma+1)}$.
		\end{enumerate}	
	\end{lem}
Intuitively speaking, items (1)--(4) show that the cut-off stage for a~$\sigma$ has the special property that all previously scheduled jumps have either been executed at this stage or have been aborted due to initialization. The only possible exception are jumps scheduled for strategies that are lexicographically smaller than~$\sigma$; however, those strategies will never be applied again in any case. Since there exist infinitely many strategies along the true path with the properties required in Definition~\ref{cut_off_def} there exist infinitely many cut-off stages. Then item~(5) implies that
$x$~is regainingly approximable.	
\begin{kor}
	The number $x$ is regainingly approximable.\qed
\end{kor}

We finish with the proof of Lemma~\ref{cut_off_properties}, completing 
the proof of the existence of a nearly computable regainingly approximable number that is not computable.
\goodbreak
	\begin{proof}[Proof of Lemma~\ref{cut_off_properties}] \;
		\begin{enumerate}
			\item If we have $t_\sigma = t_1$, then the desired jump at $t_1$ is made immediately. By construction, all strategies $\tau \in \SigmaS$ with $\sigma \prefix \tau$ or $\sigma <_L \tau$ are initialized at stage $t_\sigma$. Otherwise, we have $t_\sigma > t_1$ and the desired jump at $t_1$ is split and scheduled. 
			Remember that then, by construction, the last of these split jumps will ultimately be executed by the shortest $\rho \in \SigmaS$ with $\rho0 \prefixeq \sigma$ at stage $t_\sigma$. 
			In other more formal words, $\rho$ must be applied and expansionary at $t_\sigma$ with $c(\rho)[t_\sigma] = \left\langle \sigma, 1 \right\rangle$. Then again, by construction, all strategies $\tau \in \SigmaS$ with $\sigma \prefix \tau$ or $\sigma <_L \tau$ are initialized at stage $t_\sigma$.
						
			\item Let $\tau \in \SigmaS$; we distinguish three cases:
			\begin{itemize}
				\item Assume $\sigma <_L \tau 0$. This is equivalent to $\sigma <_L \tau$. Due to (1), $\tau$ is initialized at $t_\sigma$, so we have $c(\sigma)[t_\sigma+1] = 0$.
				\item Assume $\sigma \prefix \tau 0$. This is equivalent to $\sigma \prefixeq \tau$. If $\sigma \prefix \tau$, then, due to~(1), $\tau$ is initialized at $t_\sigma$, so we have $c(\sigma)[t_\sigma+1] = 0$. 
				
				Otherwise, we have $\sigma = \tau$. By choice of $t_0$ we have $c(\sigma)[t_0+1]=0$. We claim that, from stage $t_0+1$ onward up to stage $t_\sigma$, strategy $\sigma0$ is never applied; from this it follows in particular that for all stages~$t$ with $t_0 < t \leq t_\sigma+1$ we have $c(\sigma)[t]=0$, and we are done.
				To see that the claim is true, note that 
				\begin{itemize}
					\item in stages $t_0+1$ up to $t_1-1$, by definition, $\sigma$ cannot be expansionary, thus  $\sigma0$ cannot be applied by construction;
					\item in stage $t_1$, by assumption, $\sigma$ is applied and threatened, thus  $\sigma0$ cannot be applied by construction;
					\item if we are in the case where $t_\sigma > t_1$ then, by construction, in stages $t_1+1$ up to stage $t_\sigma$ we have that  $\sigma$ and thus in particular $\sigma0$ are never applied.
				\end{itemize}

			\goodbreak
			
				\item Assume $\tau 0 \prefixeq \sigma$. By Fact~\ref{fact:counters-on-expansionary-stages}, we must have $c(\tau)[t_1] = 0$. 
				If $t_{\sigma} = t_1$, then $\sigma$ is applied and threatened at $t_\sigma$ and the jump resulting from this threat is executed immediately; thus $\tau$'s counter remains unchanged and so $c(\tau)[t_{\sigma} + 1] = c(\tau)[t_1] = 0$.
				
				\goodbreak
				
				So assume $t_{\sigma} > t_1$, meaning that the jump resulting from the threat against $\sigma$ at stage $t_1$ is split and scheduled for later execution.	
				If there is a stage $t_2$ with $t_1 \leq t_2 \leq t_\sigma$ at which $c(\tau)[t_2]=\langle \sigma, k \rangle$ for some $k\in \IN$ (meaning that $\tau$ plays a part in executing the split jumps for~$\sigma$), then by construction there must also be a latest stage $t_3$ with $t_2 \leq t_3\leq t_\sigma$ where $c(\tau)[t_3]=\langle \sigma, 1 \rangle$ und $c(\tau)[t_3+1]=0$ (after which $\tau$ has done its part for $\sigma$). By construction, between $t_3+1$ and~$t_\sigma$ no extensions of $\tau0$ can be applied (because all remaining stages from~$t_3+1$ until~$t_\sigma$, if any, are used by strict prefixes of $\tau$ to execute the split jumps that were delegated upwards by $\tau$), and therefore we have  \[c(\tau)[t_{\sigma} + 1] = c(\tau)[t_3+1]=0.\]
				If, on the other hand, there is no stage $t_2$ as above, then $\tau$ was ``skipped over'' by $\sigma$ in the sense that when $\sigma$ was threatened at stage $t_1$, there existed a prefix $\psi$ with $\psi0 \prefixeq \tau0 \prefixeq \sigma$ and some $k\in \IN$ such that $c(\psi)[t_1+1]$ was set to $\langle \sigma, k\rangle$ directly (meaning that $\tau$ plays \textit{no} part in executing the split jumps for~$\sigma$). In this case, none of $\tau0$'s extensions can ever be applied in stages $t_1$ up to $t_\sigma$. Thus they can never set $\tau$'s counter to any non-zero value, and we have $c(\tau)[t_{\sigma} + 1] = c(\tau)[t_1] = 0$.
			\end{itemize}
			In summary, the implication $c(\tau)[t_\sigma + 1] > 0 \Rightarrow  \tau0 <_L \sigma$ holds.
						
			\item Let $\tau \in \SigmaS$ with $\tau 0 <_L \sigma$. We claim that if
			there were a stage $t > t_\sigma$ at which $\tau$ is applied and expansionary then $\sigma$ would be initialized at $t$, and since $t >t_0$ this contradicts the assumptions. To see this, note that there are two possible behaviours at stage $t$ when $\tau$ is applied and expansionary: 
			\begin{itemize}
				\item If $\tau0$ is applied in the next substage of stage $t$, then that stage must settle on some extention $\rho$ of $\tau0$. Thus, $\rho <_L \sigma$, and $\sigma$ is initialized at the end of stage $t$;
				\item if, on the other hand, we have $c(\tau)[t]=\langle \alpha , k \rangle$ for some $\alpha \suffixeq \tau0$ and some $k \in \IN$, then  stage $t$ ends with the initialization of every strategy to the right of $\alpha$, and this includes $\sigma$.
			\end{itemize}				

			\item Since $\sigma$ is not initialized after $t_\sigma$, no $\tau <_L \sigma$ can be applied and no $\tau \prefix \sigma$ can be applied and threatened after $t_\sigma$. Due to $t_\sigma \geq t_1$, $\sigma$ is not applied and threatened after $t_\sigma$ either. Thus, if some strategy $\tau$ is applied and threatened at some stage after $t_\sigma$, we must have $\sigma \prefix \tau$ or $\sigma <_L \tau$.
			\item Define $R := \{\tau \in \SigmaS \mid \sigma \prefix \tau \text{ or } \sigma <_L \tau \}$. 
			According to~(2), (3) and~(4), all jumps that are made after stage $t_\sigma$ must be due to fresh threats to strategies $\tau \in R$ when they are applied; that is, formally speaking, if we write $I := \{t \in J \mid t \geq t_\sigma + 1\}$ then we have $t_\sigma+1 \leq\min u(I)$.
			
			By~(1), all $\tau \in R$ are initialized at $t_\sigma$, thus ${w(\tau)[t_\sigma+1] = \nu(\tau) + t_\sigma + 2}$. As~$\lambda \notin R$, we have $\nu(\tau) \geq 1$ for all $\tau \in R$. 
			Thus, using
			Corollary~\ref{jsjhfarjslkdgssfg} with~$t_1=t_\sigma + 1$ and $t_2 \to \infty$ as well as Lemma~\ref{sdfjkasdjlkfgjkleknjjxvc}, we  obtain
	\begin{align*}
		x - x_{t_\sigma + 1}
		&= \sum_{t \in I} \; (x_{t+1} - x_t) \\
		&\leq \sum_{\mathclap{t' \in u(I)}} \; 2^{-w(\truepath[t'])[t']} \\
	&= \sum_{\tau \in R} \;\;\sum_{t' \in u(I)\colon \tau = \truepath[t']} \!\!\!\!\!\!\!\!\!\! 2^{-w(\tau)[t']} \\
	&\leq \sum_{\tau \in R} \;2^{-w(\tau)[t_\sigma + 1]+1} \\
	&= \sum_{\tau \in R} \;2^{-(\nu(\tau) + t_\sigma + 1)} \\
	&\leq 2^{-(t_\sigma+1)}.\qedhere
\end{align*}\end{enumerate}\end{proof}\end{proof}


\section{Incomparability}

In this section we prove that  the regainingly approximable and the nearly computable numbers are incomparable within the left-computable numbers. One of the separations is provided by the following proposition.
\begin{prop}\label{sadjasjhkrtjhqwejhxvbbmasdgh}
	There exists a regainingly approximable number which is not nearly computable.
\end{prop}
\begin{proof}
	A result of Downey, Hirschfeldt, and LaForte~\cite[Theorem~2.15]{DHL01} implies that every strongly left-computable number that is nearly computable is, in fact, computable.
	But Hertling, Hölzl, and Janicki~\cite{HHJ2023} established the existence of regainingly approximable numbers that are strongly left-computable without being computable. Such a number can therefore not be nearly computable.
\end{proof}

The other separation is the second main result of this article. 
\begin{theorem}\label{satz:FBNAA}
	There exists a left-computable number which is nearly computable but not regainingly approximable.
\end{theorem}
While the proof of Theorem~\ref{satz:FBNAA} will be structurally similar to that of Theorem~\ref{satz:fast-berechenbar-aufholend-approximierbar}, there are important differences that, while seemingly subtle, vastly impact the dynamics of the construction. In these introductory explanations we will thus focus on the differences between both constructions, but will then proceed with the construction and verification in full detail.
We again prove the theorem by constructing a computable non-decreasing sequence of rational numbers $(x_t)_t$ which converges nearly computably to some $x$.
To guarantee that $x$ is not regainingly approximable, we can use the following characterization.
\begin{lem}[Hertling, Hölzl, Janicki~\cite{HHJ2023}]\label{prop:charakterisierung-aa}
	The following statements are equivalent for a left-computable number~$x$:
	\begin{enumerate}
		\item $x$ is regainingly approximable.
		\item For every computable non-decreasing sequence of rational numbers $(x_n)_n$ converging to $x$ there exists a computable increasing function $s \colon \IN \to \IN$ with $x - x_{s(n)} < 2^{-n}$ for infinitely many $n \in \IN$.
	\end{enumerate} 
\end{lem}
It is thus sufficient to ensure that for the sequence $(x_t)_t$ that we are constructing there is no computable function $s$ as in the second item. Therefore, for every~$e \in \IN$, we define the following negative requirement:
\begin{equation*}
	\mathcal{N}_e \colon \varphi_e \text{ total and increasing } \Rightarrow \left(\exists m \in \IN\right) \left(\forall n \geq m\right) \;  x - x_{\varphi_e(n)} \geq 2^{-n}
\end{equation*}
Notice how the quantifiers differ from those in the negative requirements used to prove Theorem~\ref{satz:fast-berechenbar-aufholend-approximierbar}; here, a negative requirement can be threatened infinitely often, which will need to be reflected in the details of the construction below.

To ensure that $x$ is nearly computable, we use the same positive requirements as before; namely, for every $e \in \IN$:
\begin{equation*}
	\mathcal{P}_e \colon \varphi_e \text{ total and increasing }  \Rightarrow \left(x_{\varphi_e(t+1)} - x_{\varphi_e(t)}\right)_t
	\text{ converges computably to } 0
\end{equation*}
As in the last proof, the two types of requirements seemingly contradict each other. We will again resolve this apparent conflict by splitting large jumps into small jumps that will be scheduled and delayed for later execution.

We point out that, despite the more demanding negative requirements used here, the proof of this theorem is somewhat easier than that of Theorem~\ref{satz:fast-berechenbar-aufholend-approximierbar}. This is because in that proof, unlike here, in order to establish regaining approximability, much effort had to be devoted to ensuring  the existence of the special cut-off stages.

\smallskip

We again work with an infinite injury priority construction on an infinite binary tree of strategies~$\sigma \in \Sigma^*$ that each are responsible for satisfying both requirements $\mathcal{N}_{\left|\sigma\right|}$ and $\mathcal{P}_{\left|\sigma\right|}$, and define 
the \emph{parameters} of such a strategy as the following four functions from $\SigmaS \times \IN$ to $\IN$, defined for every $\sigma \in \Sigma^*$ and every~$t\in\IN$:
\begin{itemize}
	\item a \textit{counter} $c(\sigma)[t]$
	\item a \emph{pause flag} $p(\sigma)[t]$
	\item a \textit{restraint} $r(\sigma)[t]$
	\item a \textit{witness} $w(\sigma)[t]$		
\end{itemize}
As before, by construction, $w$ and $r$ will be non-decreasing in $t$ for each $\sigma$. In the construction, some negative requirements will require attention infinitely often. The pause flag will be used to prevent  high priority negative requirements from precluding other requirements from ever receiving attention.

\smallskip

After these informal remarks, we proceed with the full proof of the theorem.	
\begin{proof}[Proof of Theorem~\ref{satz:FBNAA}]
	We recursively define a computable non-decreasing sequence of rational numbers $(x_t)_t$ starting with $x_0 := 0$. At the same time, we also recursively define four functions ${c, p, r, w \colon \SigmaS \times \IN \to \IN}$
	starting with
\begin{align*}
	c(\sigma)[0]	&:= 0, \\
	p(\sigma)[0]	&:= 0, \\
	r(\sigma)[0]	&:= 0, \\
	w(\sigma)[0]	&:= \nu(\sigma),
\end{align*}
for all $\sigma \in \SigmaS$.
As in the proof of Theorem~\ref{satz:fast-berechenbar-aufholend-approximierbar}, the construction proceeds in stages consisting of at most $t + 1$ substages. 
At the end of each stage $t\in \IN$, we will fix a finite binary string $\truepath[t]$ and say that stage~$t$ {\em settles on}~$\truepath[t]$. Again, this string is determined in the substages as follows: In the first substage of each stage we {\em apply} 
strategy~$\lambda$; and in each substage when some $\sigma$ is applied, we can choose whether we want to {\em apply} $\sigma0$ or $\sigma1$ in the following substage or whether we want to let the current stage end right after the current substage. 
The strategy applied in the last substage of a stage is then the string that that stages settles on.

We define the same terms as before, but with the different conditions needed here:
\begin{itemize}
	\item  \emph{Initializing} a strategy~${\tau \in \SigmaS}$ at a stages $t$ means setting
\begin{align*}
	c(\tau)[t+1] &:= 0, \\
	w(\tau)[t+1] &:= \nu(\tau) + t + 2.
\end{align*}
	\item We say that a strategy $\sigma\in \SigmaS$ is {\em threatened at stage $t$} if the  conditions 
	\begin{itemize}
	\item $p(\sigma)[t] = 0$,
	\item $\ell(e)[t] \geq w(\sigma)[t]$,
	\item $x_{t} - x_{\varphi_{e}(\ell(e)[t])} < 2^{-w(\sigma)[t]}$
\end{itemize}
	are satisfied. As before, to defeat this threat, we would like to react by making some large jump.
	
	\item We say that a strategy $\sigma\in \SigmaS$ is \emph{expansionary at stage $t$} if the conditions
	\begin{itemize}
		\item $\ell(e)[t]\geq 0$ where $e:=|\sigma|$,
		\item $x_{t} - x_{\varphi_{e}(\ell(e)[t])} < 2^{-r(\sigma)[t]}$
	\end{itemize}
	are satisfied. As before this means that $\varphi_{e}$ has made some progress towards looking like a total, increasing function. 
\end{itemize}

We complete the description of the construction by giving the details of what we do in a substage of stage~$t$ when a strategy $\sigma \in \SigmaS$ is applied; note the significantly different initialization behaviour compared with the proof of Theorem~\ref{satz:fast-berechenbar-aufholend-approximierbar}:
\begin{enumerate}
	\item Let $e := \left|\sigma\right|$. If we have $e = t$, then this is the last substage, we set
	\begin{equation*}
		x_{t+1} := x_t,
	\end{equation*}
	we initialize all strategies $\tau \in \SigmaS$ with $\sigma <_L \tau$, and terminate stage~$t$. Otherwise we continue with~(2).	

	\item \emph{Negative requirement:}
	If $\sigma$ is not threatened at stage $t$, set
	\[p(\sigma)[t+1] := 0,\] meaning in particular that if strategy $\sigma$ was paused before, it is now unpaused again. Then jump directly to~(3).
	
	\smallskip
		
 Otherwise, that is if $\sigma$ is threatened at stage $t$, check if there is a $\gamma \in \Sigma^*$ with $\gamma 0 \sqsubseteq \sigma$ and $r(\gamma)[t+1] \geq w(\sigma)[t]$. If not, then let
	\begin{align*}
		x_{t+1} &:= x_{t} + 2^{-w(\sigma)[t]}.
	\end{align*}
If yes, then choose the longest such $\gamma$ and set
	\begin{align*}
		c(\gamma)[t+1] 	&:= \left\langle \sigma, 2^{r(\gamma)[t+1] - w(\sigma)[t]}\right\rangle, \\
		x_{t+1} &:= x_t.
	\end{align*}
	In either case, this will be the last substage of this stage. We pause strategy~$\sigma$ by setting
	\begin{align*}
		p(\sigma)[t+1] &:= 1;		
	\end{align*}
	this prevents $\sigma$ from being threatened in the next stage and thus gives other, lower priority strategies a chance to act. We also set 
\begin{align*}
	w(\sigma)[t+1] &:= w(\sigma)[t] + 1
\end{align*}
	which, unlike in the proof of Theorem~\ref{satz:fast-berechenbar-aufholend-approximierbar}, is necessary due to the infinitary nature of the negative requirements. We then initialize all strategies $\tau \in \SigmaS$ with $\sigma <_L \tau$, and terminate stage $t$.
	
	\item \emph{Positive requirement:}
	If $\sigma$ is not expansionary at stage $t$, continue with the next substage $e+1$ applying~$\sigma 1$. As before, this
 means that we currently have no reason to believe that $\varphi_e$ is a total and increasing function.
		
		Otherwise, if $\sigma$ is expansionary at stage $t$, we check if $c(\sigma)[t] = 0$. If this is the case, then we set
	\begin{equation*}
		r(\sigma)[t+1] := r(\sigma)[t] + 1
	\end{equation*} 
	and continue with the next substage $e+1$ applying~$\sigma 0$. As before, the ``0'' documents that we
	have made our restraint on future jump sizes stricter because we currently consider $\varphi_e$ a viable candidate for being a total and increasing function.
	
	Otherwise there exist a strategy $\alpha \in \SigmaS$ and a number $k \in \IN$ with $c(\sigma)[t] = \left\langle \alpha, k+1\right\rangle$, meaning that we still have scheduled jumps to execute. Thus check if there exists a $\gamma \in \Sigma^*$ with $\gamma 0 \sqsubseteq \sigma$. If not, then set 
	\begin{equation*}
		x_{t+1} := x_{t} + 2^{-r(\sigma)[t]}.
	\end{equation*}
	If yes, choose the longest such $\gamma$. The argument used to prove Fact~\ref{restrait_monotony} works exactly in the same way for the present construction; thus we have $r(\gamma)[t+1] \geq r(\sigma)[t]$ and we can set
	\begin{align*}
		c(\gamma)[t+1] &:= \left\langle \alpha, 2^{r(\gamma)[t+1] - r(\sigma)[t]}\right\rangle, \\
		x_{t+1} &:= x_t.
	\end{align*}
	In either case, this will be the last substage of this stage. As one of the scheduled jumps has been taken care of, we can decrement the corresponding counter by setting
	\begin{equation*}
			c(\sigma)[t+1] := \begin{cases}
				0 &\text{if $k = 0$,} \\
				\left\langle \alpha, k\right\rangle &\text{otherwise}.
			\end{cases}
	\end{equation*}
	Then we initialize all strategies $\tau \in \SigmaS$ with $\sigma0 <_L \tau$, and terminate stage~$t$.
\end{enumerate}

	If some of the parameters $c(\sigma)[t+1]$, $p(\sigma)[t+1]$, $r(\sigma)[t+1]$, or $w(\sigma)[t+1]$ have not been explicitly set to some new value during stage $t$, then we set them to preserve their last respective values $c(\sigma)[t]$, $p(\sigma)[t]$, $r(\sigma)[t]$, or $w(\sigma)[t]$. 

\medskip

We proceed with the verification. The following properties of the restraint and the witness functions follow directly from their definitions; we omit the proofs.
\goodbreak
\begin{fact}\label{satz:FBNAA:lem04}
For all $\sigma, \tau \in \Sigma^*$ and all $t \in \IN$ we have
\begin{itemize}
	\item \makebox[\widthof{$w(\sigma)[t]$}][r]{$r(\sigma)[t]$} $\leq t$,
	\item \makebox[\widthof{$w(\sigma)[t]$}][r]{$r(\sigma)[t]$} $\leq r(\sigma)[t+1]$,
	\item $w(\sigma)[t]$ $\leq w(\sigma)[t+1]$.\qed
\end{itemize}
\end{fact}

	It is easy to verify that  Fact~\ref{asdjltzhjkqwehk} and Fact~\ref{fact:counters-on-expansionary-stages} carry over to this construction by literally the same arguments.
	The next two statements hold because the pause function is not affected by initializations.
	\begin{fact}\label{sdfjhasdfkjlersdvsdfdfgdfssd}
		Let $t_1 < t_2$ be two consecutive stages at which some strategy $\sigma \in \SigmaS$ is applied. Then $p(\sigma)[t_1]=0$ or $p(\sigma)[t_2]=0$.\qed
	\end{fact}
	\begin{fact}\label{sdfjhasdfkjldfghjkasjlersdvsdf}
	Let $t_1 < t_2$ be two consecutive stages at which some strategy $\sigma \in \SigmaS$ is applied. Assume that $\sigma$ is threatened at $t_1$. Then $\sigma$ is not threatened at $t_2$.
	\end{fact}
\begin{proof}
	Since $\sigma$ is threatened at $t_1$, by construction we have both $p(\sigma)[t_1] = 0$ and $p(\sigma)[t_1+1] = 1$. Then, by construction, we still have $p(\sigma_l)[t_2] = 1$, and thus $\sigma$ cannot be threatened at~$t_2$.
\end{proof}
The proof of Lemma~\ref{satz:FBNBAA:lem02} has to be modified as follows to reflect the different dynamics of the present construction.
	\begin{lem}\label{fghsdfdhdfsdgdfghgfsfsdf}
		The following statements are equivalent for any $\sigma \in \Sigma^*$:
		\begin{enumerate}
			\item The strategy $\sigma$ is applied and expansionary at infinitely many stages.
			\item The strategy $\sigma$ is applied, not threatened, and expansionary at infinitely many stages.
			\item Infinitely many stages settle on some extension of $\sigma0$.
		\end{enumerate}
	\end{lem}
	\begin{proof}
		``$(1) \Rightarrow (2)$'': Let $t_1$ be any stage at which $\sigma$ is applied and expansionary. We claim there is a stage $t_2 \geq t_1$ at which $\sigma$ is applied, not threatened, and expansionary. If $\sigma$ is not threatened at $t_1$ we can let $t_2:=t_1$. Otherwise let $t_2$ be the next stage at which $\sigma$ is applied. In this case, to see that $t_2$ is as needed, note that by construction and by the definition of an expansionary stage, $\sigma$ is still expansionary at $t_2$. But by Fact~\ref{sdfjhasdfkjldfghjkasjlersdvsdf}, $\sigma$ cannot be threatened at~$t_2$.
		
		\smallskip
		
		``$(2) \Rightarrow (3)$'':  Suppose that there are infinitely many stages at which $\sigma$ is applied, not threatened, and expansionary. Let $t_0 \in \IN$ be an arbitrary  such stage. 
		We claim that there must be a stage $t^\ast \geq t_0$ where $\sigma0$ is applied. 
		If $c(\sigma)[t_0] = 0$, then $t^\ast=t_0$ by construction. 
		Otherwise fix $\alpha \in \SigmaS$ and $k \geq1$ with $c(\sigma)[t_0] = \left\langle \alpha, k\right\rangle$. Then let $t_1, \dots, t_k \in \IN$ denote the $k$ consecutive stages where $\sigma$~is applied, not threatened, and expansionary that immediately follow $t_0$; we claim that  we can find~$t^\ast$ among them. Namely, if between $t_0$ and $t_k$ an initialization of $\sigma$ occurs at some stage $\widehat t$, then let $t^\ast$ be the smallest element of $\{t_0,\dots, t_k\}$ that is greater than~$\widehat t$. Otherwise, if no such initialization occurs, then we must have  $c(\sigma)[t_k] = 0$ by construction, and $t^\ast = t_k$.
		
		\smallskip
		
		``$(3) \Rightarrow (1)$'': Suppose that infinitely many stages settle on some extension of $\sigma0$. By construction $\sigma$ must be applied and expansionary at these stages.
	\end{proof}
	
	Define $J$ and $u$ as in the proof of Theorem~\ref{satz:fast-berechenbar-aufholend-approximierbar}. We continue by noting that Fact~\ref{sdlakjldfgahrjehrfhksbfasjd234}, Lemma~\ref{satz:FBNBAA:lem:spruenge-expansionary}, Fact~\ref{remark:spruenge-expansionary-weak}, Lemma~\ref{satz:FBNBAA:lem:spruenge-threatened}, Fact~\ref{remark:spruenge-threatened-weak}, and 
	Corollary~\ref{jsjhfarjslkdgssfg} again carry over to this construction with literally the same proofs. The proof of Lemma~\ref{sdfjkasdjlkfgjkleknjjxvc} needs to be modified as follows, because here we use different rules for increasing witnesses.	
	\begin{lem}\label{sdfjkasdjlkfgjkleknjjxvcdfdfgdfg}
		Let $I \subseteq J$, let $\sigma \in \SigmaS$ and let $t_0 \leq \min u(I)$. Then we have
		\begin{equation*}
			\sum_{\mathclap{t' \in u(I)\colon \truepath[t']=\sigma}}
			\; 2^{-w(\sigma)[t']} \leq 2^{-w(\sigma)[t_0]+1}.
		\end{equation*}
	\end{lem}
	\begin{proof}
		By definition of $u$, all the numbers $t'$ that appear in the sum on the left-hand side are stages where $\sigma$ is applied and threatened. By construction, between two such stages,  the value of $\sigma$'s witness grows by at least $1$.
		Thus,	
		\begin{equation*}
			\sum_{\mathclap{t' \in u(I)\colon \truepath[t']=\sigma}}
			\; 2^{-w(\sigma)[t']}
			\leq \sum_{k=0}^{\infty} \;
			2 ^{-(w(\sigma)[t_0]+k)}
			\leq 2^{-w(\sigma)[t_0]+1}. \qedhere
		\end{equation*}
	\end{proof}
	Using these tools, we can prove the following statement.
	\begin{prop}
		The sequence $(x_t)_t$ is computable, non-decreasing, and bounded from above. Thus, its limit $x := \lim_{t\to\infty} x_t$ is a left-computable number.
	\end{prop}
	
	\begin{proof}
		The proof is literally the same as that of Proposition~\ref{satz:FBNBAA:prop:konvergenz}, except with Lemma~\ref{sdfjkasdjlkfgjkleknjjxvcdfdfgdfg} in place of Lemma~\ref{sdfjkasdjlkfgjkleknjjxvc}.
	\end{proof}

The following is analogous to Proposition~\ref{sdfjlw3ljsdadfgjrfasdfdfg}, but allows for an easier proof.
\begin{prop}\label{satz:FBNAA:lem03}
For every $l \in \IN$, there exists a binary string $\sigma_l \in \Sigma^{l}$ satisfying the following conditions:
\begin{enumerate}
	\item There exist infinitely stages in which $\sigma_l$ is applied.
	\item There exist only finitely many stages in which $\sigma_l$ is initialized.
\end{enumerate}
In particular, the map $t \mapsto \left|\truepath[t]\right|$ is unbounded, and if we let 
$\truepath$ denote the true path of $(\truepath[t])_t$ then  conditions (1) and (2) hold for every $\sigma \prefix \truepath$.
\end{prop}
\begin{proof} 
	We proceed by induction. The claim trivially holds for $l = 0$ since $\lambda$ is applied in every stage but never initialized.
	Assume that for some fixed $l \in \IN$ there exists a strategy $\sigma_l \in \Sigma^l$ satisfying both conditions. Let $t_0 \in \IN$ be the earliest stage after which $\sigma_l$ is no longer initialized. 
	There are two cases:
	\begin{itemize}
		\item {\em $\sigma_l0$ is applied infinitely often:}
		We claim $\sigma_l0$ is never initialized after $t_0$; this is because by construction any initialization of $\sigma_l0$ would have to occur in a stage where some $\gamma <_L \sigma_l0$ is applied; this would also initialize $\sigma_l$, contradiction. Thus we can choose $\sigma_{l+1} := \sigma_l 0$.
		
		\item {\em $\sigma_l0$ is applied only finitely often:}
		Due to Lemma~\ref{fghsdfdhdfsdgdfghgfsfsdf} in this case there is a stage $t_1 \geq t_0$ after which $\sigma_l$ is never again applied and expansionary. Then, by construction, $\sigma_l1$ is never again initialized after $t_1$.
 We also claim that for every $t_2 \geq t_1$ at which $\sigma_l$ is applied there exists a stage $t_3\geq t_2$ at which $\sigma_l1$ is applied. If $\sigma_l$ is not threatened at $t_2$ then by construction $\sigma_l1$ is applied at $t_2$, and we are done. Otherwise, if $\sigma_l$ is threatened at $t_2$, then by Fact~\ref{sdfjhasdfkjldfghjkasjlersdvsdf} we have that $\sigma_l$ cannot be threatened at the next stage $t_3 > t_2$ where $\sigma_l$ is applied. Thus $\sigma_l1$ is applied at $t_3$. 
Consequently, we can choose $\sigma_{l+1} := \sigma_l 1$.
	\end{itemize}
The second part of the proposition is proven in literally the same way as in the proof of Proposition~\ref{sdfjlw3ljsdadfgjrfasdfdfg}.
\end{proof}

The following lemma demonstrates that the dynamics that result from the way the construction was set up work as intended.
\begin{lem}\label{dfjkhasdjldfgmjasddh}
	Let $e$ be such that $\varphi_e$ is total and increasing and let $\sigma \prefix\truepath$ with $|\sigma|=e$. Then the following statements hold:
	\begin{enumerate}
		\item There exist infinitely many stages at which $\sigma$ is applied and threatened, and in particular, $(w(\sigma)[t])_t$ tends to infinity. Furthermore, if $t_0$ is a stage after which $\sigma$ is never again initialized, then for ${m := w(\sigma)[t_0]}$ we have that every $k \geq m$ is an element of $(w(\sigma)[t])_t$.
		
		\item There exist infinitely many stages at which $\sigma$ is applied and expansionary.
		In particular, $(r(\sigma)[t])_t$ tends to infinity.
	\end{enumerate}
\end{lem}
\begin{proof} \; \nopagebreak
\begin{enumerate}
	\item 	 By Proposition~\ref{satz:FBNAA:lem03}, $\sigma$ is applied infinitely often and there is a stage $t_0$ after which $\sigma$ is not initialized anymore. Then, for $t> t_0$ we can only have $w(\sigma)[t+1] > w(\sigma)[t]$ if $\sigma$ is applied and threatened at $t$. 
	Recall that the assumption that $\varphi_e$ is total and increasing is equivalent to the statement that 
$(\ell(e)[t])_t$ tends to infinity. 
	Thus, by the definition of a threatened stage, by the fact that $(x_t)_t$ converges, and using Fact~\ref{sdfjhasdfkjlersdvsdfdfgdfssd}, for every $t_1> t_0$ there exists a $t_2 \geq t_1$ at which $\sigma$ is applied and threatened and such that $w(\sigma)[t_1+1] = w(\sigma)[t_1]+1$. This is enough to establish all three claims.
	
	\item Let $t_1 \in \IN$ be arbitrary. 
	By (1), there are infinitely many stages at which $\sigma$~is applied and threatened. Thus, by Fact~\ref{sdfjhasdfkjldfghjkasjlersdvsdf}, there must also be infinitely many stages, where $\sigma$ is applied and not threatened. 
	Therefore, and since $(x_t)_t$ converges, there is a smallest $t_2 \geq t_1$ such that for all 
	$t \geq t_2$ we have 
	${x_t - x_{\varphi_{e}(\ell(e)[t])} < 2^{-r(\sigma)[t_1]}}$
	and at which $\sigma$ is applied, not threatened, and by by definition expansionary.
	Thus, since $t_1$ was arbitrary, there are infinitely many stages at which $\sigma$ is applied and expansionary. Then by Lemma~\ref{fghsdfdhdfsdgdfghgfsfsdf}, $\sigma0$ is applied at infinitely many stages, and by construction, $(r(\sigma)[t])_t$ tends to infinity.\qedhere
\end{enumerate}	
\end{proof}

Noting that Lemma~\ref{satz:FBNBAA:lem07} also holds for this construction with literally the same proof as before, we are ready to prove that all negative requirements are satisfied.
\begin{prop}
	$\mathcal{N}_e$ is statisfied for all $e \in \IN$.
\end{prop}
\begin{proof} 
	Let $e$ be such that $\varphi_e$ is total and increasing, and let $\sigma \prefix\truepath$ with $|\sigma|=e$. 	Let $t_0 \in \IN$ be the earliest stage after which $\sigma$ is not initialized anymore. Let $m$ be as in Lemma~\ref{dfjkhasdjldfgmjasddh}~(1), and let $n \geq m$ be arbitrary. Then there exists a uniquely determined stage $t_1 \geq t_0$ at which $\sigma$ is applied and threatened and such that $w(\sigma)[t_1] = n$. Let $t_2 > t_1$ be the next stage at which $\sigma$ is applied. Then 
	\[
		x - x_{\varphi_{e}(n)} = x - x_{\varphi_{e}(w(\sigma)[t_1])} \\
		\geq x_{t_2} - x_{\varphi_{e}(\ell(e)[t_1])}  \\
		\geq   2^{-w(\sigma)[t_1]} \\
		= 2^{-n};
	\]
	here, the first inequality uses the fact that 
	$\ell(e)[t_1] \geq w(\sigma)[t_1]$  by definition of a threatened stage, and the second inequality uses Lemma~\ref{satz:FBNBAA:lem07}.

	In summary, there exists a number $m \in \IN$ such that for all $n \geq m$ we have $x - x_{\varphi_{e}(n)} \geq 2^{-n}$. Therefore,  $\mathcal{N}_e$ is satisfied.
\end{proof}
\begin{kor} 
	The number $x$ is not regainingly approximable.\qed
\end{kor}

It remains to prove that $x$ is nearly computable. 
Write 
\[S := \{e \in \IN \mid \varphi_{e} \text{ is total and increasing}\}.\] The following lemma is analogous to Lemma~\ref{satz:FBNBAA:lem08}; however, the infinitary nature of the negative requirements and the different initialization strategy used in this construction are reflected in the higher complexity of its statement and proof.
\begin{lem}\label{satz:FBNAA:lem09}
	 Let $\sigma \prefix \truepath$ and $t_0 \in \IN$ be the earliest stage after which $\sigma$ is no longer initialized and after which there are no more stages at which any $\tau \prefixeq \sigma$ with $\left|\tau\right| \notin S$ is applied and threatened. Let $t_1, t_2 \in \IN$ with $t_0 \leq t_1< t_2$ be two consecutive stages at which 
	$\sigma$ is applied and expansionary. Then we have
	\begin{equation*}
		x_{t_2} - x_{t_1} \leq 2^{-r(\sigma)[t_1] + 1} + \sum_{\mathclap{\substack{\phantom{,}\tau \prefixeq \sigma, \\ \left|\tau\right| \in S}}} 2^{-w(\tau)[t_1]+1}.
	\end{equation*}
\end{lem}
As in the proof of Lemma~\ref{satz:FBNBAA:lem08}, we need to analyze all possible causes for jumps being made between stages $t_1$ and $t_2$. Again there are essentially two such causes: we may still have jumps scheduled from before stage $t_1$ that await execution, and new threats may arise after $t_1$ leading to further jumps. 
The most important difference to Lemma~\ref{satz:FBNBAA:lem08} is however that here initial segments $\tau$ of $\sigma$ with $\left|\tau\right| \in S$ may also get threatened --- possibly multiple times --- leading to additional jumps that need to be accounted for by the sum in the statement of the lemma.
\begin{proof}
	We again distinguish between obligations to make jumps that may already stand at the moment when we enter stage $t_1$, and new obligations to make jumps that are created between stages $t_1$ and $t_2$.
	First, to quantify jumps due to potential standing obligations, note that in substages $0$ up to $|\sigma|-1$ of stage $t_1$, no jumps are made by construction. 
	Now consider the following facts:
	\begin{enumerate}
		\item No strategy $\tau <_L \sigma$ will be applied again after stage $t_1$, as otherwise this would lead to initialization of $\sigma$, which contradicts our assumptions. Thus, if any jumps are still scheduled for any such $\tau$ (that is, if $c(\tau)[t_1]>0$), we need not take them into account.
		
		\item For strategies $\tau \in \SigmaS$ with 
		$\tau 0 \prefixeq \sigma$ and $|\tau|\notin S$
		we must have  $c(\tau)[t_1] = 0$ due to Fact~\ref{fact:counters-on-expansionary-stages}. Furthermore,  $\tau$ will never be applied and threatened between stages $t_1$ and $t_2$ by our assumptions. Thus we need not consider these strategies in the remainder of the proof.
		
		\item For strategies $\tau \in \SigmaS$ with 
		$\tau 0 \prefixeq \sigma$ and $|\tau|\in S$
		we must have ${c(\tau)[t_1] = 0}$, as well. However, such a $\tau$ may get threatened between stages $t_1$ and $t_2$ which might lead to jumps that we need to take into account.
				
		\item While for strategies $\tau \in \SigmaS$ with $\tau1 \prefixeq \sigma$ we might have $c(\tau)[t_1]>0$, none of these scheduled jumps will ever be executed from $t_1$ onwards. This is because
		such a jump could only be executed during a stage at which $\tau$ is applied and expansionary. This would by construction lead to the initialization of all $\rho$ with $\tau0 <_L \rho$, which in particular include $\sigma$; contradiction. Thus, as before, we need not take such $\tau$'s into account.
		
		\item No strategy $\tau$ with $ \sigma0 <_L \tau$ can be applied at stage $t_1$ by construction. While such $\tau$'s may be applied during some stage $t_1 < t < t_2$, note that we set $c(\tau)[t_1+1]=0$ for all of them when they are initialized at the end of stage $t_1$.
		
		\item For a strategy $\tau$ with $\sigma0 \prefixeq \tau$ we may indeed have $c(\tau)[t_1] > 0$.
		
		\item Similarly, it might be the case that $c(\sigma)[t_1] > 0$.
	\end{enumerate}
	Thus, any jumps made in stages $t_1$ up to $t_2 -1$ are either caused by some strategy as in~(6) or~(7), or must be due to \textit{new} threats to strategies of type~(3) or~(5) that occur strictly between substage~$|\sigma|$ of stage~$t_1$ and stage $t_2$. We establish upper bounds for all of these cases:
	\begin{itemize}
		\item Concerning strategies of type~(6), by construction, if we make any jump at all while applying them, then these can only occur immediately at stage $t_1$ in substage $|\tau|$.
		We claim that such a jump can not occur due to  $\tau$ being applied and expansionary at $t_1$ with ${c(\tau)[t_1] > 0}$. This is because such a jump would not be made at $t_1$, but would by construction instead be split and scheduled for later execution by some strategy $\rho$ with $\sigma0 \prefixeq \rho0 \prefixeq \tau$, and this strategy $\rho$ will not be applied before stage $t_2$ by construction.
		Thus, the only reason to make a jump in this case is if
		$\tau$ is applied and threatened at $t_1$. Then, by construction the jump made would be of size 
		\[2^{-w(\tau)[t_1]} < 2^{-r(\sigma)[t_1+1]} \leq 2^{-r(\sigma)[t_1]}.\]
		
		\item Concerning~(7), in this case, by construction, we would like to make a jump of size $2^{-r(\sigma)[t_1]}$. The jump may be made immediately, or split and scheduled for later execution by some strategies that are prefixes of $\sigma$; but in either case the total sum of all jumps resulting from this is bounded by $2^{-r(\sigma)[t_1]}$ by Lemma~\ref{satz:FBNBAA:lem:spruenge-expansionary}.
		
		\item  Consider new threats to strategies $\tau$ of type~(5) that occur strictly between substage $|\sigma|$ of stage $t_1$ and stage $t_2$.
		Since these $\tau$'s were all initialized at stage $t_1$, we have $w(\tau)[t_1+1] = \nu(\tau) + t_1 + 2$ for each of them.
		Thus, 
		applying Lemma~\ref{sdfjkasdjlkfgjkleknjjxvcdfdfgdfg} to each such $\tau$ and then summing over all of them, the total sum of all jumps made or scheduled when such $\tau$'s are applied in stages between $t_1$ and $t_2$ can be at most $2^{-t_1}$, which by Fact~\ref{satz:FBNBAA:lem03} is at most~$2^{-r(\sigma)[t_1]}$.	
		
		\item Finally, consider new threats to strategies $\tau$ of type~(3) that occur at some stage $t$ strictly between substage $|\sigma|$ of stage $t_1$ and stage $t_2$. By Fact~\ref{remark:spruenge-threatened-weak} 		
		each such threat can only contribute jumps totaling at most $2^{-w(\tau)[t]}$. And, by construction, whenever this occurs, we set $w(\tau)[t+1]=w(\tau)[t]+1$. Thus, even if the same $\tau$ is threatened multiple times between $t_1$ and $t_2$, the total sum of all jumps associated with these threats sums to at most~${2^{-w(\tau)[t_1]+1}}$.		
	\end{itemize}
By construction, cases~(6) and~(7) exclude each other, thus the jumps caused by cases (5)--(7) combined contribute at most $2^{-r(\sigma)[t_1]+1}$. Together with the maximally possible contribution of all $\tau$'s of type~(3) we obtain \begin{equation*}
	x_{t_2} - x_{t_1} \leq 2^{-r(\sigma)[t_1] + 1} + \sum_{\mathclap{\substack{\phantom{,}\tau \prefixeq \sigma, \\ \left|\tau\right| \in S}}} 2^{-w(\tau)[t_1]+1}.\qedhere
\end{equation*}
\end{proof}

	After these preparations, we can show that all positive requirements are satisfied.
	\begin{prop}
		$\mathcal{P}_e$ is statisfied 	for all $e \in \IN$.
	\end{prop}
	\begin{proof} 
		Let $e \in \IN$ and $\sigma \prefix \mathcal{T}$ with $\left| \sigma \right| = e$. Suppose that $\varphi_{e}$ is  total and increasing. Then, by Lemma~\ref{dfjkhasdjldfgmjasddh}~(2), $(r(\sigma)[t])_{t}$ tends to infinity. Let $t_0 \in \IN$ be the earliest stage after which $\sigma$ is no longer initialized and such that, for all $\tau \prefixeq \sigma$ with $\left|\tau\right| \notin S$, there are no more stages at which $\tau$ is applied and threatened.
		For $n\in\IN$, write  
		\[t(n):=\min\left\{t \geq t_0\colon\; 
			\parbox{9cm}{\centering
				$r(\sigma)[t] \geq n+3$ and $\sum_{\tau \prefixeq \sigma, \left|\tau\right| \in S} 2^{-w(\tau)[t] + 1} \leq 2^{-(n+1)}$\\
				and $\sigma$ is applied and expansionary at stage $t$
			}
		\right\}.\]
		
		Note that for each $\tau \prefixeq \sigma$ with $\left|\tau\right| \in S$ we have that $(w(\tau)[t])_t$ tends to infinity by Lemma~\ref{dfjkhasdjldfgmjasddh}~(1), and thus $t(n)$ is defined for every $n$.
		Thus, we can define a function $v \colon  \IN \to \IN$ via ${v(n) = \ell(e)[t(n)]}$ for all $n\in\IN$. Clearly, $t$ and $v$ are computable. 
		
	 We claim that $v$ is a modulus of convergence of the sequence $(x_{\varphi_{e}(t+1)} - x_{\varphi_{e}(t)})_t$. 
	 To see this, let $n\in \IN$ and $i \geq v(n)$. Let $t_2 > t_1 \geq t(n)$ be the two consecutive stages at which $\sigma$ is applied and expansionary with $\ell(e)[t_1] \leq i$ and $\ell(e)[t_2] \geq i+1$. Applying Lemma~\ref{satz:FBNAA:lem09} and the assumption that $\sigma$ is expansionary at $t_1$, we obtain	 
		\begin{align*}
			x_{\varphi_{e}(i+1)} - x_{\varphi_{e}(i)} &\leq
			x_{\varphi_{e}(\ell(e)[t_2])} - x_{\varphi_{e}(\ell(e)[t_1])} \\
			&= \underbrace{\left( x_{\varphi_{e}(\ell(e)[t_2])} - x_{t_2} \right)}_{\leq 0} + \left( x_{t_2} - x_{t_1} \right) + \left( x_{t_1} - x_{\varphi_{e}(\ell(e)[t_1])} \right) \\
			&\leq  \left(2^{-r(\sigma)[t_1]+1} + \sum_{\substack{\tau \prefixeq \sigma, \\ \left|\tau\right| \in S}} 2^{-w(\tau)[t_1]+1}\right) 
			+ 2^{-r(\sigma)[t_1]} \\
			&\leq  2^{-r(\sigma)[t(n)]+1} + 2^{-(n+1)} + 2^{-r(\sigma)[t(n)]} \\
			&< 2^{-r(\sigma)[t(n)]+2} + 2^{-(n+1)} \\
			&\leq 2^{-n}.
		\end{align*}
		Thus $(x_{\varphi_{e}(t+1)} - x_{\varphi_{e}(t)})_t$ converges computably to zero, and $\mathcal{P}_e$ is satisfied.	
	\end{proof}
	
	\begin{kor}
		The number $x$ is nearly computable.\qed
	\end{kor}
	
This completes 
	the proof of the existence of a left-computable, nearly computable number that is not regainingly approximable.
	\phantom\qedhere
\end{proof}
\section{Non-speedable non-randoms}

In this final section we prove our third main result by showing that there exists a left-computable number that is neither speedable nor Martin-Löf random; thereby giving a negative answer to the question of Merkle and Titov~\cite{MT2020} whether among the left-computable numbers the Martin-Löf randoms are characterized by being non-speedable. The understanding of the relationship between near computability and regaining approximability that we gained in the last section will be instrumental for this.
We begin by establishing a more convenient characterization of speedability.
\begin{lem}\label{lem:characterisierung-speedable}
	A left-computable number is speedable if and only if there exists a constant $\rho \in (0, 1) $ and a computable increasing sequence $(x_n)_n$ of rational numbers converging to $x$ such that there are infinitely many $n \in \IN$ with $\frac{x_{n+1} - x_n}{x - x_n} \geq \rho$.
\end{lem}
\begin{proof}
	Suppose that $x$ is speedable, that is, by definition, there exists a constant $\rho' \in (0, 1)$ and a computable increasing sequence $(x_n)_n$ converging to $x$ with $\frac{x - x_{n+1}}{x - x_n} \leq \rho'$ for infinitely many $n \in \IN$. 
	If we let $\rho := 1 - \rho'$, this is equivalent to $\frac{x_{n+1} - x_n}{x - x_n} = 1 - \frac{x - x_{n+1}}{x - x_n} \geq \rho$ for infinitely many $n \in \IN$.
\end{proof}
Thus, a left-computable number $x$ is non-speedable if and only if for every computable increasing sequence of rational numbers $(x_n)_n$ converging to~$x$ the sequence $\left(\frac{x_{n+1} - x_n}{x - x_n}\right)_n$ converges to zero. Before we use this characterization we first examine the relationship between speedable and regainingly approximable numbers.
\begin{prop}\label{prop:aufholend-approximierbar-impliziert-beschleunigbar}
	Every regainingly approximable number is speedable.
\end{prop}
\begin{proof}
	Let $x$ be regainingly approximable. Then, by definition, there exists a computable non-decreasing sequence of rational numbers $(x_n)_n$ converging to $x$ with $x - x_n < 2^{-n}$ for infinitely many $n \in \IN$. Define the sequence $(y_n)_n$ by $y_n := x_n - 2^{-n}$ for all $n \in \IN$. This sequence is computable, increasing and also converges to $x$. Then, for every $n \in \IN$ with $x - x_n < 2^{-n}$, we have
	\begin{equation*}
		\frac{y_{n+1} - y_n}{x - y_n} = \frac{\left(x_{n+1} - x_n\right) + 2^{-(n+1)}}{\left(x - x_n\right) + 2^{-n}} > \frac{2^{-(n+1)}}{2^{-n} + 2^{-n}} = \frac{1}{4}.\qedhere
	\end{equation*}
\end{proof}
The converse is not true as the two following known results imply.
\begin{prop}[Merkle, Titov \cite{MT2020}]
	Every strongly left-computable number is speedable.
\end{prop}
\begin{theorem}[Hertling, Hölzl, Janicki \cite{HHJ2023}]\label{satz:SLBNAA}
	There exists a strongly left-computable number that is not regainingly approximable.
\end{theorem}
\begin{kor}\label{dfsdjhhdbfajfkjgsdjfhdasnfbsafevcvssda}
	There exists a speedable number that is not regainingly approximable.\qed 
\end{kor}

Thus, in general, the regainingly approximable numbers are a proper subset of the speedable numbers. However, as we will show now, the two notions become equivalent once we restrict ourselves to nearly computable numbers. To demonstrate this, we employ the following two convenient characterizations of the left-computable nearly computable and of the left-computable regainingly approximable numbers.
\begin{prop}[Hertling, Janicki~\cite{HJ2023}]\label{prop:characterisierung-lnc}
	The following statements are equivalent for a left-computable number~$x$:
	\begin{enumerate}
		\item $x$ is nearly computable.
		\item For every computable non-decreasing sequence of rational numbers $(x_n)_n$ converging to $x$, the sequence $(x_{n+1} - x_n)_n$ converges computably to zero.
	\end{enumerate}
\end{prop}
\begin{prop}[Hertling, Hölzl, Janicki \cite{HHJ2023}]\label{prop:characterization-reg-app}
		The following statements are equivalent for a left-computable number~$x$:
		\begin{enumerate}
			\item $x$ is regainingly approximable.
			\item There exists a computable non-decreasing sequence of rational numbers $(x_n)_n$ converging to $x$ and a computable non-decreasing and unbounded function $f \colon \IN \to \IN$ with
			$x - x_n \leq 2^{-f(n)}$
			for infinitely many $n \in \IN$.
		\end{enumerate}
	\end{prop}
\begin{theorem}\label{dsfjknsdkjldfgljsdfdfgdfterter}
	Let $x$ be a speedable number which is nearly computable. Then $x$ is  regainingly approximable.
\end{theorem}
\begin{proof}
	Choose a constant $\rho \in (0,1)$ and a computable increasing sequence of rational numbers $(x_n)_n$ converging to $x$ that witness the speedability of $x$ in the sense of Lemma~\ref{lem:characterisierung-speedable}.
	
	Since $x$ is nearly computable, the sequence $(x_{n+1} - x_n)_n$ converges computably to zero, as witnessed by some modulus of convergence $f \colon \IN \to \IN$ that is computable, non-decreasing and unbounded.
	Define $g \colon \IN \to \IN$ via
	\begin{equation*}
		g(n) := \begin{cases}
			0 &\text{if $f(0) > n$}, \\
			\max\{k \in \IN \colon f(k) \leq n\} &\text{otherwise},
		\end{cases} 
	\end{equation*}
	for all $n \in \IN$. Clearly, $g$ is computable, non-decreasing and unbounded, and we have $x_{n+1} - x_n < 2^{-g(n)}$ for all $n \geq f(0)$.
	Fix some $k \in \IN$ with $\nicefrac{1}{\rho} \leq 2^k$ and define $h \colon \IN \to \IN$ via
	\begin{equation*}
		h(n) := \max\{0, g(n) - k\}
	\end{equation*}
	 for all $n \in \IN$. Again, $h$ is computable, non-decreasing and unbounded.
	  
	Let $m := \min\{i \geq f(0) \colon g(i) \geq k \}$ and consider any of the infinitely many $n \geq m$ for which $\frac{x_{n+1} - x_n}{x - x_n} \geq \rho$ holds by choice of $\rho$ and of $(x_n)_n$. For all of these $n$ we have
	\begin{equation*}
		x - x_n \leq \nicefrac{1}{\rho} \cdot \left(x_{n+1} - x_n\right) < 2^k \cdot 2^{-g(n)} = 2^{-h(n)}.
	\end{equation*}
	Hence, using Proposition~\ref{prop:characterization-reg-app}, 	
	$x$ is regainingly approximable.
\end{proof}

Putting together Proposition~\ref{prop:aufholend-approximierbar-impliziert-beschleunigbar} and Theorem~\ref{dsfjknsdkjldfgljsdfdfgdfterter}, we obtain the following corollary.

\goodbreak

\begin{kor}
		Let $x$ be a left-computable number that is nearly computable. Then the following statements are equivalent:
		\begin{enumerate}
			\item $x$ is regainingly approximable.
			\item $x$ is speedable.\qed
		\end{enumerate}
	\end{kor}
	
	\goodbreak

Together with Theorem~\ref{satz:FBNAA} we obtain the next corollary.
\begin{kor}\label{kor:LCNCNS}
	There exists a left-computable number which is nearly computable but not speedable.\qed
\end{kor}
Stephan and Wu~\cite{SW2005} showed that a left-computable number which is nearly computable cannot be Martin-Löf random. Together with Corollary~\ref{kor:LCNCNS} this implies our final main result, a negative answer to the question of Merkle and Titov~\cite{MT2020}.
\begin{kor}
	There exists a left-computable number which is not speedable and not Martin-Löf random.\qed
\end{kor}
It remains unclear whether a counterexample could also be found \textit{outside} the set of nearly computable numbers; accordingly, we conclude the article with the following open question.
\begin{opqn}
	Do the Martin-Löf random numbers, the nearly computable numbers, and the speedable numbers together form a covering of all left-computable numbers?
\end{opqn}

\section{Acknowledgments}

The authors would like to thank Ivan Titov for helpful discussions.

\bibliography{cca}
\bibliographystyle{abbrv}

\end{document}